\documentclass[preprint,12pt]{elsarticle}

\usepackage{algorithm} 
\usepackage{algorithmicx}
\usepackage{algpseudocode}

\usepackage{color}
\usepackage{amssymb}
\usepackage{mathrsfs}
\usepackage{amsmath}
\usepackage{amsthm}
\newtheorem{theorem}{Theorem}[section]
\newtheorem{lemma}[theorem]{Lemma}
\newtheorem{proposition}[theorem]{Proposition}
\newtheorem{corollary}[theorem]{Corollary}
\newtheorem{definition}[theorem]{Definition}
\newtheorem{assumptions}[theorem]{Assumptions}
\newdefinition{remark}{Remark}

\usepackage{caption}
\journal{a journal}

\begin{document}

\begin{frontmatter}



\title{RNN-BSDE method for high-dimensional fractional backward stochastic differential equations with Wick-Itô integrals}


\author[label1]{Chunhao Cai} 
\ead{caichh9@mail.sysu.edu.cn}

\author[label1]{Cong Zhang\corref{cor1}}
\ead{congzhang402@gmail.com}

\cortext[cor1]{Corresponding author}

\address[label1]{School of Mathematics (Zhuhai), Sun Yat-sen University, Zhuhai 519082, Guangdong, People's Republic of China}

\begin{abstract}
Fractional Brownian motions(fBMs) are not semimartingales so the classical theory of Itô integral can't apply to fBMs. Wick integration as one of the 
applications of Malliavin calculus to stochastic analysis is a fine definition for fBMs. We consider the fractional forward backward stochastic differential equations(fFBSDEs) 
driven by a fBM that have the Hurst parameter in $(\frac{1}{2}, 1)$ where $\int_{0}^{t} f_s \, dB_s^H$ is in the sense of a Wick integral, and relate our fFBSDEs to the system
of partial differential equations by using an analogue of the Itô formula for Wick integrals. And we develop a deep learning algorithm referred to as 
the RNN-BSDE method based on recurrent neural networks which is exactly designed for solving high-dimensional fractional BSDEs and their corresponding partial differential equations.
\end{abstract}



\begin{keyword}
Fractional Brownian motions\sep Wick calculus \sep Fractional backward stochastic differential equations \sep Deep learning \sep  Recurrent neural networks 


\end{keyword}

\end{frontmatter}



\section{Introduction}
\label{sec1}

In recent years, deep learning has been developed and has been widely used to deal with high-dimensional problems about Partial Differential
Equations(PDEs), though the key points of the idea of deep learning which include convolutional neural networks\cite{21701}
and back propagation\cite{1986Learning} were already developed by 1990. The Recurrent Neural Network algorithm, which is classical and basic for dealing with time series, 
was already developed by 2000\cite{ELMAN1990179}\cite{JORDAN1997471}. In the 21st century, the advanced hardware and datasets rather than 
theories are mainly the reasons that deep learning takes off.

We have already known some work on solving PDEs by deep learning algorithms like PINNs\cite{raissi2019physics} and neural operator\cite{kovachki2023neural}. In this paper, we'd like to 
discuss more on solving SDEs by deep learning, which may be rare and have less work on it. The late 2010s and 2020s have witnessed some deep learning algorithms to solve 
stochastic control problems\cite{2016Deep}\cite{ji2022solving} and solve the backward stochastic differential equations(BSDEs)\cite{2017Deep}. Since the connection between nonlinear 
PDEs and BSDEs has been proved\cite{10.1007/BFb0007334}, Solving BSDEs means we can also solve the corresponding PDEs by using the same algorithm.

These work has inspired us whether it will work if we try to solve fractional backward stochastic differential equations(fBSDEs) by the deep learning method 
since it has been proved to be effective for BSDEs. In our work, we consider the following fractional forward backward stochastic differential equations(fFBSDEs)
\begin{equation}\label{eqn1}
\begin{cases}
dX_s = \mu(s,X_s)ds + \sigma (s,X_s)dB_s^H,\\
-dY_s = f(s,X_s,Y_s,Z_s) ds - Z_s dB_s^H,\\
X_t = x,\\
Y_T = g(X_T),
\end{cases} 
\end{equation}
where $B_s^H$ is a fractional Brownian motion(fBM), let Hurst constant $H \in (\frac{1}{2}, 1)$, $t$ and $T$ mean the start and final time, $\left\{(X_s, Y_s, Z_s), t \le s \le T\right\} $ are all stochastic processes.

Mandelbrot and van Ness defined a fractional Brownian Motion as that
\begin{definition}[fBM\cite{fbmdef}]\label{def1.1}
Let $B_0$ an arbitrary real number. We call $B_t^H$ a fractional Brownian Motion with Hurst parameter H and starting value $B_0$ at time 0, such as
\begin{enumerate}
\item $B^H(0,\omega)= B_0$,
\item $B^H(t,\omega)-B^H(0,\omega)= \frac{1}{\Gamma(H+\frac{1}{2})}\int_{0}^{t}(t-s)^{H-\frac{1}{2}}dB(s,\omega)$.
\end{enumerate}
\end{definition}

Since the fractional Brownian motions are not semimartingales, $\int_{}^{}f \, dB^H$ is not well-defined in the sense of Itô integral\cite{le2016brownian}. In this paper, we instead consider it as the Wick integral\cite{ksendal1996ANIT}\cite{2000Duncan}\cite{HU2011}\cite{mishura2008stochastic}. To introduce it, the Wick product is used. Moreover, the stochastic calculus including the Wick integration and differentiation which we will introduce later is based on the Malliavin calculus\cite{ksendal1996ANIT}.

\begin{sloppypar}
Our purpose is to approximate the $\mathscr{F}_t^H-$adapted processes $\left\{(Y_s, Z_s), t \le s \le T\right\}$(where $\mathscr{F}_t^H$ is the nature filtration of$B^H$) such that
\end{sloppypar}
\begin{equation}\label{eqn3}
Y_t =  g(X_T) + \int_{t}^{T} f(s,X_s,Y_s,Z_s)\, ds - \int_{t}^{T} Z_s\, dB_s^H,\quad a.s.
\end{equation} 
by using the deep learning method.

Our paper is organized as follows: In Sec.\ref{sect2}, we summarize some results from the applications of Malliavin calculus to stochastic analysis that we need for our work. In Sec.\ref{sect3}, we study existence and uniqueness of the solution of equation \eqref{eqn3} and consider the relationship between fFBSDEs and nonlinear PDEs, 
also we present some calculations which are useful for the numerical experiments in Sec.\ref{sect5}. In Sec.\ref{sect4}, we introduce our deep learning algorithm for solving fFBSDEs which is based on 
the Recurrent Neural Network(especially, the LSTM can be also a choice for our method as a kind of RNNs), and we refer to this algorithm as RNN-BSDE method. In Sec.\ref{sect5}, we give some numerical experiments of solving some
parabolic PDEs by RNN-BSDE method including mRNN-BSDE based on a stacked-RNN and LSTM-BSDE based on a multi-layer LSTM and compare our method with other methods that are mainly designed for solving BSDEs rather than fractional BSDEs.

In this paper, any $\int_{}^{}f_s\,dB_s^H$ is in the fractional Wick Itô Skorohod integral(fWIS) sense.
\color{black}
\section{Wick calculus}\label{sect2}
Before we introduce the Malliavin calculus applied to the fractional Brownian motion, we would like to intepret why we choose the fWIS integrals rather than 
the pathwise integrals. Afterall, we only consider $\frac{1}{2} \le H <1$. But the fWIS integrals based on Wich calculus has some positive meaning in finance.
The nonexistence of arbitrage in a market is a basic equilibrium condition, however, pathwise fBm market the existence of arbitrage was proved by \cite{pathwise}.
It is not possible to make a sensible mathematical theory for a market with arbitrage. On the other hand, there is no strong arbitrage in the WIS fractional Black
Scholes market which is proved by \cite{biagini_stochastic_2008}. In \cite{HU2011}, the authors has proved their 
Itô fractional Black-Scholes market is complete and computed explicitly the price and replicating portfolio of a European option in this market.
We refer to \cite{biagini_stochastic_2008} to learn more details about the pathwise fBm market and WIS fractional Black-Scholes market. 

\color{black}
We will review some results of the applications of Malliavin calculus to stochastic analysis. For the details one can learn more about stochastic calculus for Brownian motions from \cite{ksendal1996ANIT}\cite{2019Introduction} and the version for fractional Brownian motions from 
\cite{2000Duncan}\cite{HU2011}\cite{mishura2008stochastic}. Firstly we have some preparations.

Set $\frac{1}{2} \le H \le 1$ and fix the Hurst constant. Define
\begin{equation}\label{eqn4}
\phi(s,t) = H(2H-1)\left|s-t\right|^{2H-2}, \quad s,t \in \mathbb{R}.
\end{equation}

Let $f: \mathbb{R} \to \mathbb{R}$ be measurable. We say that $f$ belongs to the Hilbert space $L^2_{\phi}(\mathbb{R})$ if
\begin{equation}\label{eqn5}
{\left|f\right|}_\phi^2 := \int_{\mathbb{R}}^{}\int_{\mathbb{R}}^{} \phi(s,t) f(s) f(t)\, dsdt < \infty.
\end{equation}
The inner product on $L^2_{\phi}(\mathbb{R})$ is denoted by ${(\cdot,\cdot)}_\phi$.

For any $f \in L^2_{\phi}(\mathbb{R})$, define $\varepsilon(f)$ as
\begin{equation}\label{eqn6}
\varepsilon(f) = \exp{(\int_{\mathbb{R}}^{} f \, dB^H-\frac{1}{2}\left\lvert f \right\rvert^2_\phi)}.
\end{equation}
$\varepsilon(f)$ is called an exponential function and let $E$ be the linear span of \\$\left\{\varepsilon(f),f \in L^2_{\phi}(\mathbb{R})\right\}$.

Consider the fractional white noise probability space denoted $(\Omega, \mathscr{F}, P) = (S^\prime(\mathbb{R}), \mathscr{B}, P_\phi)$ where $S^\prime(\mathbb{R})$ is the dual space of 
the Schwartz space $S(\mathbb{R})$ and the probability measure $P_\phi$ exists by using the Bochner-Minlos theorem(see e.g.\cite{spde2010}) such that 
\begin{equation}
\label{eqn2.1e}
\int_{\Omega}^{} e^{i<\omega, f>}\, d P_\phi(\omega) = e^{-\frac{1}{2}{\left|f\right|}_\phi^2 } \qquad for \quad all \quad f \in S(\mathbb{R}),
\end{equation} 
It follows from \eqref{eqn2.1e} that 
\begin{equation*}
E_{P_\phi}\left(\left\langle \cdot,f\right\rangle \right) = 0, \quad E_{P_\phi}\left(\left\langle \cdot,f\right\rangle ^ 2\right) = {\left|f\right|}_\phi^2.   
\end{equation*}
So that $\tilde{B}_t^H$ can be defined as an element of $L^2(P_\phi)$ such that
\begin{equation*}
\tilde{B}(t,\omega) = <\omega, \chi_{\left[0, t\right]}\left(\cdot\right)>,   
\end{equation*}
then by Kolmogorov's continuity theorem, $\tilde{B}_t^H$ has a continuous version $B_t^H$. From now on, we endow $\Omega$ with the natural
filtration $\mathscr{F}_t^H$ of $B_t^H$.

Then what we would introduce is the Hermite polynomials which are
\begin{equation}
\label{eqn2.2e}
h_n(x) = (-1)^ne^{x^2/2}\frac{d^n}{dx^n}\left(e^{-x^2/2}\right),\quad n=0,1,2,\dots
\end{equation}

Set 
\begin{equation}
e_n(u) = \Gamma_\phi^{-1}\left(\tilde{h}_n\right)(u) \quad u \in \mathbb{R},    
\end{equation}
where $\tilde{h}_n(x)$ is the Hermite functions that are defined as
\begin{equation*}
\tilde{h}_n(x) = \pi^{-1/4}\left(\left(n-1\right)! \right)^{-1/2} h_{n-1}(\sqrt{2}x)e^{-x^2/2},  
\end{equation*}
and 
\begin{equation*}
\begin{aligned}
\Gamma_\phi f(u) = c_H \int_{u}^{\infty}(t-u)^{H-3/2}f(t)\,dt,\\
c_H = \sqrt{\frac{H(2H-1)\Gamma(3/2-H)}{\Gamma(H-1/2)\Gamma(2-2H)}},
\end{aligned}    
\end{equation*}
where $\Gamma$ denotes the gamma function.
In view of Lemma 3.1 of \cite{HU2011}, 
$\{e_i\}_{i=1}^{\infty}$ is an orthonormal basis of $L^2_{\phi}(\mathbb{R})$ such that for any $t\in\mathbb{R}$ there exists $C_t < \infty$ such that
\begin{equation*}
\left\lvert \int_{\mathbb{R}}^{}e_n(s)\phi(s, t)\,ds\right\rvert < C_tn^{1/6}.
\end{equation*}

Let $\mathcal{I} = (\mathbb{N}_0^\mathbb{N})_c$ denote the set of all (finite) multi-indices $\alpha = (\alpha_1, \alpha_2, \dots, \alpha_m)$ of non-negative
integrals. Then 
\begin{equation*}
\mathcal{H}_\alpha(\omega):= h_{\alpha_1}\left(\left\langle \omega, e_1\right\rangle \right)\cdots h_{\alpha_m}\left(\left\langle \omega, e_m\right\rangle \right).
\end{equation*}
Finally, we can introduce the fractional Wiener-Itô chaos expansion theorem\cite{2000Duncan} after these preparations,
\begin{theorem}[The fractional Wiener-Itô chaos expansion]
\label{fwie}    
Let $F \in L^2(P_\phi)$. Then there exist constants $c_\alpha \in \mathbb{R}$, $\alpha \in \mathcal{I}$, such that
\begin{equation}
\label{eqn2.3e}
F(\omega) = \sum_{\alpha \in \mathcal{I}}c_\alpha \mathcal{H}_\alpha(\omega)\quad(convergence \quad in \quad L^2(P_\phi)).
\end{equation}
Moreover,
\begin{equation}
\label{eqn2.4e} 
\left\lVert F\right\rVert^2_{L^2(P_\phi)} = \sum_{\alpha \in \mathcal{I}}\alpha!c_\alpha^2,
\end{equation}
where $\alpha! = \alpha_1!\alpha_2!\cdots\alpha_m!$ if $\alpha = (\alpha_1, \dots, \alpha_m) \in \mathcal{I}$.
\end{theorem}
Then it's ready to define the fractional Hida test function and distribution spaces:
\begin{definition}[\cite{HU2011}]
(a)(The fractional Hida test function spaces)Define $(S)_H$ to be the set of all    
\begin{equation}
\begin{aligned}
\psi(\omega) &= \sum_{\alpha \in \mathcal{I}}a_\alpha\mathcal{H}_\alpha(\omega)\in L^2(P_\phi)\quad such \quad that\\
\left\lVert \psi \right\rVert^2_{H, k}&:=\sum_{\alpha \in \mathcal{I}}\alpha!a_\alpha^2(2\mathbb{N})^{k\alpha} < \infty \quad for\quad all \quad k\in\mathbb{N},
\end{aligned}    
\end{equation}
where 
\begin{equation*}
(2\mathbb{N})^\gamma=\prod_j (2j)^{\gamma_j}\qquad if \quad \gamma = (\gamma_1, \gamma_2, \dots, \gamma_m)\in \mathcal{I}.
\end{equation*}
(b)(The fractional Hida distribution spaces) Define $(S)^*_H$ to be the set of all
formal expansions
\begin{equation}
G(\omega) = \sum_{\beta \in \mathcal{I}}b_\beta\mathcal{H}_\beta(\omega) 
\end{equation}
such that
\begin{equation*}
\left\lVert G \right\rVert^2_{H, -q}:=\sum_{\beta \in \mathcal{I}}\beta!b_\beta^2(2\mathbb{N})^{-q\beta} < \infty \quad for\quad some \quad q\in\mathbb{N}. 
\end{equation*}
\end{definition}
Equip $(S)_H$ with the projective topology and $(S)^*_H$ with the inductive topology, then $(S)^*_H$ can be identified with the dual of $(S)_H$ and the action 
of $G\in(S)^*_H$ on $\psi \in (S)_H$ is given by 
\begin{equation}
\langle\langle G, \psi \rangle\rangle := \left\langle G, \psi\right\rangle _{(S)_H, (S)^*_H} := \sum_{\alpha \in \mathcal{I}}\alpha! a_\alpha b_\alpha.
\end{equation}

\begin{definition}[Wick product in $(S)^*_H$]
Let 
\begin{equation}
F(\omega) =  \sum_{\alpha \in \mathcal{I}}a_\alpha\mathcal{H}_\alpha(\omega) \quad and \quad G(\omega) =  \sum_{\beta \in \mathcal{I}}b_\beta\mathcal{H}_\beta(\omega)    
\end{equation}
be two members of $(S)^*_H$. Then we define the Wick product F $\diamond$ G of F and G by
\begin{equation}
\left(F \diamond G\right)(\omega) = \sum_{\alpha, \beta \in \mathcal{I}}a_\alpha b_\beta \mathcal{H}_{\alpha+\beta}(\omega) = \sum_{\gamma \in \mathcal{I}}\left(\sum_{\alpha+\beta=\gamma}\right)\mathcal{H}_\gamma(\omega).
\end{equation}
\end{definition}

Immediatly, we have the fractional white noise $W_H(t) \in (S)^*_H$ at time $t$ is defined by
\begin{equation}
W_H(t) = \sum_{i=1}^{\infty}\left[\int_{\mathbb{R}}e_i(v)\phi(t,v)\,dv\right]\mathcal{H}_{\epsilon^{(i)}}(\omega),     
\end{equation}
where $\epsilon^{(i)}:=(0,\dots,0,1,0,\dots,0)$ denotes the $i$th unit vector.
And we have
\begin{equation}
\int_{0}^{t}W_H(s)\,ds = B_{H}(t).
\end{equation}

Then, 
\begin{definition}[Fractional Wick Itô Skorohod integrals]
Suppose $Y:\mathbb{R} \to (S)^*_H$ is a given function such that 
$Y(t)\diamond W_H(t)$ is integrable in $(S)^*_H$. Then we define $Y$ is $dB^{(H)}-$integrable and the integral is 
\begin{equation}
\int_{\mathbb{R}}^{}Y(t)\,dB^H(t) := \int_{\mathbb{R}}^{}Y(t)\diamond W_H(t)\,dt. 
\end{equation}
\end{definition}

\subsection{Stochastic derivative}
\begin{definition}
Let $F:S^\prime(\mathbb{R})\to \mathbb{R}$ be a given function and let $\gamma\in S^\prime(\mathbb{R})$. We say that $F$ has a directional derivative in the direction $\gamma$ if
\begin{equation}
D^{(H)}_\gamma F(\omega) := \lim_{\epsilon \to 0}\frac{F(\omega+\epsilon \gamma)-F(\omega)}{\epsilon}    
\end{equation}
exists in $(S)^*_H$. If this is the case, we call $D^{(H)}_\gamma F$ 
the directional deviative of $F$ in the direction $\gamma$.
\end{definition}

\begin{definition}[Malliavin derivative]
We say that $F:S^\prime(\mathbb{R})\to \mathbb{R}$ is differentiable if there exists a map $\Psi:\mathbb{R}\to (S)^*_H$ such that
\begin{equation}
\Psi(t)\gamma(t) = \Psi(t,\omega)\gamma(t)\quad is \quad (S)^*_H-integrable   
\end{equation}
and
\begin{equation}
D^{(H)}_\gamma F(\omega) = \int_{\mathbb{R}}^{}\Psi(t,\omega)\gamma(t)\,dt
\end{equation}
for all $\gamma\in L^2(\mathbb{R})$.
Especially, we put
\begin{equation}
D^{(H)}_t F(\omega) := \frac{dF}{d\omega}(t,\omega):= \Psi(t,\omega).   
\end{equation}
We call $D^{(H)}_t F(\omega)$ the Malliavin derivative of $F$ at $t$.
\end{definition}

Then we introduce another type of Stochastic derivative.
The definition of $\phi-$derivative that its derivative operator is denoted as $D^\phi_t$ can be found in \cite{ksendal1996ANIT}\cite{2000Duncan}. 
For $g \in L^2_{\phi}(\mathbb{R})$, $\Phi$ is defined by
\begin{equation*}
(\Phi_g)_t = \int_{0}^{\infty} \phi(t,u) g_u \, du
\end{equation*}
Let $D_{\Phi_g}$(defined in \cite{2000Duncan}) an analogue of the directional derivative and $D^\phi_t$ is defined as
\begin{equation*}
D_{\Phi_g} F(\omega) = \int_{0}^{\infty} D^\phi_s F(\omega) g(s) \, ds,
\end{equation*}
For all $g \in L^2_{\phi}(\mathbb{R})$, where $F$ is a random variable and $F \in L^p$. Then $F$ is said to be $\phi-$differentiable.
Moreover, for a stochastic process $F$, the process is said to be $\phi-$differentiable if for each $t\in[0,T]$, $F(t,\cdot)$ is $\phi-$differentiable
and $D^\phi_s F_t$ is jointly measurable.

It may confuse readers since the $\phi-$derivative and Malliavin derivative are so similar. In this paper, they will work in different ways.
$\phi-$derivative $D^\phi_t F$ is used to help constructing fractional Wick Itô Skorohod integrals with value in $L^2$(in fact, it is 
also a kind of Malliavin derivative). The Malliavin derivative which is $D^{(H)}_t F(\omega)$ is used to prove the uniqueness and existence
of the solution of the fractional BSDE.

So for our purpose, finally, we introduce two very useful lemmas,
\begin{lemma}[\cite{Biagini}]
\label{lem0}
Suppose $Y:\mathbb{R} \to (S)^*_H$ is $dB^{(H)}-$integrable. Then
\begin{equation}
D_t^{(H)}\left(\int_{\mathbb{R}}^{}Y(s)\,dB^{(H)}(s) \right) = \int_{\mathbb{R}}^{}D_t^{(H)}Y(s)\,dB^{(H)}(s) + Y(t). 
\end{equation}
\end{lemma}

The following results of $\phi-$derivative are frequently used in this paper. 
\begin{lemma}[The chain rule]\label{lem1}
Let $f$ be a smooth function and $F:\Omega \to \mathbb{R}$ and $F$ is $\phi-$differentiable, then $f(F)$ is also $\phi-$differentiable and
\begin{equation}\label{eqn8}
D^\phi_tf(F) = f^\prime(F) D^\phi_tF.
\end{equation}
\end{lemma}
Also ,it's not hard to notice the Malliavin derivative $D^{(H)}_t F(\omega)$ has the same chain rule from their definitions, or see \cite{Kunze2013AnIT}.

\subsection{Fractional Wick Itô Skorohod integrals in $L^2$}
Note that the fractional Hida distribution space $(S)^*_H$ is an extension of $L^2(P_\phi)$, i.e. $L^2(P_\phi)\subset (S)^*_H$. 
$E$ is dense in $L^2(P_\phi)$(Theorem 3.1 of \cite{2000Duncan}).

Now, we want to define $\int_{0}^{T}f_t \, dB^H_t$, in which $f_t$ is an element in $L^2(P_\phi)$.
Consider $f \in E$ and an arbitrary partition of $[0,T]$ denoted $\pi$, $\diamond$ is the mark of the Wick product, then for the following Riemann sum:
\begin{equation}
\label{eqn37}
S(f,\pi) = \sum_{i=0}^{n-1}f_{t_i} \diamond (B_{t_{i+1}}^H-B_{t_i}^H),
\end{equation}
is well-defined.
Our goal is to make $S(f,\pi)$ a Cauchy sequence in $L^2(P_\phi)$ with some condtion. Let $\mathcal{L}(0,T)$ be the family of stochastic processes
$F$ on $[0,T]$ with the following properties: if and only if $E \left\lvert f\right\rvert ^2_\phi < \infty$, 
$F$ is $\phi-$differentiable, the trace of $(D^\phi_sF_t, 0\le s \le T, 0\le t \le T)$ exists, $E\int_{0}^{T}\left(D^\phi_sF_s\right)^2\,ds < \infty$,
and for each sequence of partitions $(\pi_n, n\in \mathbb{N})$ such that $\left\lvert \pi_n\right\rvert \to 0$
as $n \to \infty$, the quantities
\begin{equation*}
\sum_{t=0}^{n-1}E\left\{\int_{t_i^{(n)}}^{t_{i+1}^{(n)}}\left\lvert D^\phi_sF_{t_i^{(n)}}^{\pi}-D^\phi_sF_s\right\rvert \,ds\right\}^2     
\end{equation*} 
and
\begin{equation*}
E\left\lvert F^{\pi}-F \right\rvert^2_\phi 
\end{equation*}
tend to 0 as $n \to \infty$, where $\pi_n: 0 = t_0^{(n)} < t_1^{(n)} < \cdots < t_{n-1}^{(n)} < t_n^{(n)}=T$.

\begin{definition}[Fractional Wick Itô Skorohod integrals in $L^2$]\label{def2.1}
Denote $\left\lvert \pi \right\rvert = \max_i(t_{i+1}-t_i)$, define $\int_{0}^{T} f_t \, dB^H_t$ as
\begin{equation}\label{eqn7}
\int_{0}^{T} f_t \, dB^H_t = \lim_{\left\lvert \pi \right\rvert \to 0} S(f,\pi),\quad for ~ any ~ \pi.
\end{equation}
\end{definition}
For $f \in \mathcal{L}(0,T)$ , $S(f,\pi)$ will be a Cauchy sequence in $L^2 (P_\phi)$ so that it has the limit denoted by $\int_{0}^{T} f_s \,dB_s^H $. In view of Theorem 3.9 in \cite{2000Duncan}, for $f \in \mathcal{L}(0,T)$, the limit of \eqref{eqn37} satisfies
\begin{equation}
\label{eqn38}
E \left\lvert \int_{0}^{T} f(t) \, dB^H_t\right\rvert^2 = E\left\{\left(\int_{0}^{T} D^\phi_s f_s\,ds\right)^2 + \left\lvert f\right\rvert _\phi^2  \right\},  
\end{equation}
and
\begin{equation}
E\int_{0}^{T} f_s \,dB_s^H = 0.    
\end{equation}

Then we have some important theorems to introduce.
\begin{sloppypar}
\begin{theorem}[Itô formula for Wick integration]\label{thmito}
Let $X_t =  \xi + \int_{0}^{t} \mu(s,X_s)\, ds+ \int_{0}^{t} \sigma (s,X_s) \,dB_s^H$, $\sigma \in \mathcal{L}(0,T)$ and $E\sup_{0\le s \le T}|\mu_s| < \infty $.
Assume there is an $\alpha>1-H$ such that
\begin{equation*}
E\left\lvert \sigma_u-\sigma_v\right\rvert^2 \le C\left\lvert u-v\right\rvert^{2\alpha},
\end{equation*}
where $\left\lvert u-v\right\rvert \le \delta$ for some $\delta > 0$ and 
\begin{equation*}
\lim_{0\le u, v \le t,\left\lvert u-v\right\rvert \to 0}E\left\lvert D^\phi_u (\sigma_u-\sigma_v)\right\rvert^2=0.
\end{equation*}
Let $f:\mathbb{R_+} \times \mathbb{R} \to \mathbb{R}$, be a function having the first continuous derivative in its first variable and the second continuous derivative in its second variable. Assume that these derivatives are bounded.
Moreover, it is assumed that $E\int_{0}^{T}|\sigma_sD^\phi_sX_s|<\infty$, $\left(\frac{\partial f}{\partial x}(s,X_s)\sigma_s, 0\le s\le T\right)$ are in $\mathcal{L}(0,T)$. Then, for $0 \le t \le T$,
\begin{equation}\label{eqn9}
\begin{aligned}
f(t,X_t) =& f(0,\xi) + \int_{0}^{t} \frac{\partial f}{\partial s}(s,X_s)ds + \int_{0}^{t} \frac{\partial f}{\partial x}(s,X_s)\mu_s ds\\
&+ \int_{0}^{t} \frac{\partial^2 f}{\partial x^2}(s,X_s)\sigma_s D^\phi_s X_s ds\\ 
&+\int_{0}^{t} \frac{\partial f}{\partial x}(s,X_s)\sigma_s dB_s^H \quad a.s.    
\end{aligned}
\end{equation} 
\end{theorem}
\end{sloppypar}

if as $n \to \infty$, $\sum_{i = 0}^{n-1}f(t^{(n)}_i)(B^H(t^{(n)}_{i+1})-B^H(t^{(n)}_{i}))$ converges in $L^2(\Omega, \mathscr{F}, P)$ to the same limit
for all partitions $(\pi_n, n \in N)$ satisfying $\left\lvert \pi_n \right\rvert \to 0$ as $n \to \infty$, then this limit is called the
stochastic integral of Stratonovich type and the limit is denoted by $\int_{0}^{t}f_s \delta B_s^H$.

It can be found in \cite{2000Duncan}\cite{biagini_stochastic_2008} that  
\begin{theorem}
\label{thmtrans}
For $f \in \mathcal{L}(0,T)$, the following equality is satisfied:
\begin{equation}
\int_{0}^{t}f_s \,d B_s^H =\int_{0}^{t}f_s \delta B_s^H - \int_{0}^{t}D^\phi_s f_s\, ds \qquad a.s.
\end{equation}
\end{theorem}

\section{Uniqueness and existence of fractional backward SDEs and systems of PDEs}\label{sect3}
Consider the fractional white noise space $(\Omega, \mathscr{F}, P) = (S^\prime(\mathbb{R}^d), \mathscr{B}, P_\phi)$(The multidimensional
presentation $d > 1$ will be an analogue of the case in \cite{1992The} or see \cite{Biagini}), $\mathscr{F}_t^H = \sigma(B^H(s), 0 \le s \le t)$ and the fFBSDEs \eqref{eqn1} which we give in Sec.\ref{sec1}:
\begin{equation*}
\begin{cases}
dX_s = \mu(s,X_s)ds + \sigma (s,X_s)dB_s^H,\\
-dY_s = f(s,X_s,Y_s,Z_s) ds - Z_s dB_s^H,\\
X_t = x,\\
Y_T = g(X_T),
\end{cases} 
\end{equation*}
where $H \in (\frac{1}{2}, 1)$, $\mu(s,x): \mathbb{R}^+ \times \mathbb{R}^d \to \mathbb{R}^d$, $\sigma(s,x): \mathbb{R}^+ \times \mathbb{R}^d \to \mathbb{R}^{d \times d}$, and let
$\mu(s,X_s)$, $\sigma (s,X_s)$ satisfy the conditions of Theorem \ref{thmito}. $Y: \mathbb{R}^+ \times \Omega \to \mathbb{R}^k$, $Z: \mathbb{R}^+ \times \Omega \to \mathbb{R}^{k \times d}$,
$g(X_T)$ is $\mathscr{F}_T$-measurable, $f:\mathbb{R}^+ \times \mathbb{R}^d \times \mathbb{R}^k \times \mathbb{R}^{k \times d} \to \mathbb{R}^k$.

The solution of $X_t$ which solves \eqref{eqn1} have been studied by \cite{VaÅge01101996}. Also, as mentioned in \cite{VaÅge01101996}, simultaneous fulfilment of the Lipschitz and growth
conditions on the negative norms of coefficients is very restrictive, \cite{mishura2008stochastic} has discussed the uniqueness and existence of solutions of the
quasilinear equation of the form
\begin{equation*}
dX_s = \mu(s,X_s)ds + \sigma (s)X_sdB_s^H.
\end{equation*}

As for the uniqueness and existence of solution $\{y, z\}$, first of all, we assume
\begin{enumerate}
\item The following fractional SDE,
\begin{equation*}
dX_s = \mu(s,X_s)ds + \sigma (s,X_s)dB_s^H   
\end{equation*}
has a unique solution $X:\mathbb{R} \to \left((S)^*_H\right) ^d$.
\item $X_t, Y_t, Z_t$ are all $\mathscr{F}_t^H-$measurable for any $t \in [0, T]$, $Y$, $Z\in \mathcal{L}(0,T)$.
\end{enumerate}
It's a useful way to study fBSDE by relating fractional Backward SDEs with PDEs. Firstly, we give the following system of PDEs,
\begin{equation}\label{eqn10}
\begin{cases}	
\begin{aligned}
&\frac{\partial u(t,x)}{\partial t} + \mathcal{L}_{frac} u(t,x) + f(t,x,u(t,x),(\nabla u\sigma)(t,x))=0,\\
&u(T,x)=g(x),\\   
\end{aligned} 
\end{cases}  
\end{equation}
where $u: \mathbb{R_+} \times \mathbb{R}^d \to \mathbb{R}^k$ and we denote $u(t,x) = \left(u^1(t,x),\cdots,u^k(t,x)\right)^T$, and 
\begin{equation*}
\mathcal{L}_{frac} = 
\begin{pmatrix}
L_{frac} u_1\\
\cdots\\
L_{frac} u_k
\end{pmatrix},
\end{equation*} 
\begin{equation*}
L_{frac}=\sum_{i,j = 1}^{d}\sigma_i(t,x)D^\phi_t x_j \frac{\partial^2}{\partial x_i \partial x_j}+\sum_{i = 1}^{d}\mu_i(t,x)\frac{\partial}{\partial x_i}.
\end{equation*} 

We have the following theorem that
\begin{theorem}\label{thmpde}
Fix $\frac{1}{2}<H\le 1$, consider $X_t = X_0 + \int_{0}^{t} \mu(s,X_s) \,ds + \int_{0}^{t} \sigma (s,X_s) \,dB_s^H$ and let $\left\{(\mu_s, \sigma_s),0 \le s \le T\right\}$ 
satisfy the conditions of Theorem \ref{thmito}. If $u \in C^{1,2}([0,T] \times \mathbb{R}^d)$ solves equation \eqref{eqn10}, then $u(t,x) = Y^{t,x}_t$(where $X^{t,x}_t = x$), $t \ge 0$, $x \in \mathbb{R}^d$,
where $\left\{(Y_s^{t,x},Z_s^{t,x}), t \le s \le T \right\} $  is the solution of the BSDE \eqref{eqn3}.
\end{theorem}
\begin{proof}
Apply Theorem \ref{thmito} to the solution $u$ of \eqref{eqn10},
\begin{equation*}
\begin{aligned}
u(T,X_T)-u(t,X_t) =& -\int_{t}^{T} \mathcal{L}_{frac} u(s,X_s) + f \,ds\\
&+ \sum_{i = 1}^{d} \int_{t}^{T} \mu_i(s,X_s)  \frac{\partial u(s,X_s)}{\partial x^i} \,ds\\ 
&+ \sum_{i = 1}^{d} \int_{t}^{T} \sigma_i(s,X_s)  \frac{\partial u(s,X_s)}{\partial x^i} \,dB^H_s\\
&+ \sum_{i,j = 1}^{d}\int_{t}^{T} \sigma_i(s,X_s) D^\phi_s X_s^j  \frac{\partial^2 u(s,X)}{\partial x^i \partial x^j} \,ds \quad a.s.
\end{aligned}   
\end{equation*}
\begin{equation}
u(t,X_t) = g(X_T) +\int_{t}^{T}f\,ds - \sum_{i = 1}^{d} \int_{t}^{T} \sigma_i(s,X_s)  \frac{\partial u(s,X_s)}{\partial x^i} \,dB^H_s, \quad a.s.
\end{equation}
it can be seen that this theorem has been proved.
\end{proof}

What remains to us is to prove the converse to the above result, let us introduce a lemmma first,
\begin{lemma}
\label{lemma3.2}    
For any $0 \le s \le T$,
\begin{equation}
Z_s = \nabla Y_s (\nabla X_s)^{-1}\sigma(s,X_s).
\end{equation}
\end{lemma}
\begin{proof}
Note that for any $\mathscr{F}_t^H-$adapted process $F$ and $s \le t$, the stochastic deviation of $F$ satisfies
\begin{equation*}
D^{(H)}_t F_s = 0,    
\end{equation*}
which has been stated in page 57 of \cite{Nualart_Nualart_2018}(Though there the authors consider the standard Brownian motion, for the FBM, there would be nothing 
different since techniques of Malliavin calculus make classic BM share many properties with the fBM.)

Consider
\begin{equation*}
Y_s = Y_t - \int_{t}^{s}f\left(X_r, Y_r, Z_r\right)\,dr + \int_{t}^{s}Z_r\,dB_r^H, 
\end{equation*}
Let $(Z_s)_i$ denote the $i-$th column ot he matriz. In view of \ref{lem0} and the chain rule for the Malliavin derivative, for any $t < \theta \le s \le T$ and $1 \le i \le d$,
\begin{equation*}
\begin{aligned}
(D^{(H)}_\theta)_i  Y_s = &(Z_\theta)_i - \int_{\theta}^{s} [f_x^{\prime}\left(X_r, Y_r, Z_r\right)(D^{(H)}_\theta)_i X_r + f_y^{\prime}\left(X_r, Y_r, Z_r\right)(D^{(H)}_\theta)_i Y_r \\
&+ f_z^{\prime}\left(X_r, Y_r, Z_r\right)(D^{(H)}_\theta)_i Z_r]\,dr\\
&+ \int_{\theta}^{s}(D^{(H)}_\theta)_i Z_r\,dB_r^H,
\end{aligned}
\end{equation*}
Moreover, set $\theta = s$, then
\begin{equation}
\label{eqn3.4e}
D^{(H)}_s Y_s = Z_s,\quad s \quad a.e.    
\end{equation}

Next, consider the matrix valued process
\begin{equation}
\label{eqn3.3e} 
\nabla X_s = \nabla X_\theta + \int_{\theta}^{s} \partial_x \mu(r, X_r) \nabla X_r\,dr + \int_{\theta}^{s} \partial_x \sigma(r, X_r) \nabla X_r\,dB_r^H,
\end{equation}
Also, we have 
\begin{equation*}
D^{(H)}_\theta X_s = \sigma(\theta, X_\theta) + \int_{\theta}^{s} \partial_x \mu(r, X_r) D^{(H)}_\theta X_r\,dr + \int_{\theta}^{s} \partial_x \sigma(r, X_r) D^{(H)}_\theta X_r\,dB_r^H,    
\end{equation*}
as a consequence of uniqueness of the solution of the fractional SDE, for $t \le \theta \le s \le T$,
\begin{equation}
\label{eqn3.7e}
D^{(H)}_\theta X_s = \nabla X_s (\nabla X_\theta)^{-1}\sigma(\theta, X_\theta).
\end{equation}

for any $1 \le i \le d$,
\begin{equation}
\label{eqn3.5e}
\begin{aligned}
(D^{(H)}_\theta)_i Y_s =& g^\prime(X_T)(D^{(H)}_\theta)_i X_T + \int_{s}^{T} F\left(r, (D^{(H)}_\theta)_i Y_s, (D^{(H)}_\theta)_i Z_s\right) \,dr\\
&-\int_{s}^{T}(D^{(H)}_\theta)_i Z_r\,dB_r^H\\
F(r, u, v) =& f_x^{\prime}\left(X_r, Y_r, Z_r\right)(D^{(H)}_\theta)_i X_r + f_y^{\prime}\left(X_r, Y_r, Z_r\right)u + f_z^{\prime}\left(X_r, Y_r, Z_r\right)v.
\end{aligned}    
\end{equation}

Also, we have 
\begin{equation}
\label{eqn3.6e}
\begin{aligned}
&\nabla Y_s = g^\prime(X_T)\nabla X_T + \int_{s}^{T} F\left(r, \nabla Y_s, \nabla Z_s\right) \,dr -\int_{s}^{T}\nabla Z_r\,dB_r^H\\
&F(r, u, v) = f_x^{\prime}\left(X_r, Y_r, Z_r\right)\nabla X_r + f_y^{\prime}\left(X_r, Y_r, Z_r\right)u + f_z^{\prime}\left(X_r, Y_r, Z_r\right)v.
\end{aligned}    
\end{equation}
\end{proof}
Thanks to \cite{HU2009}, the linear backward stochastic equation \eqref{eqn3.5e} can be solved by the unique solution $\left\{(D^{(H)}_\theta)_i Y_s, (D^{(H)}_\theta)_i Z_s; 0 \le s \le T\right\}$.
It is deduced from the uniqueness of \eqref{eqn3.5e}, \eqref{eqn3.6e} and equation \eqref{eqn3.7e} that 
\begin{equation}
\label{eqn3.8e} 
D^{(H)}_\theta Y_s = \nabla Y_s (\nabla X_\theta)^{-1}\sigma(\theta, X_\theta)  
\end{equation}
for $t \le \theta \le s \le T$.

Furthermore, 
\begin{equation}
\label{eqn3.9e}
Z_s = D^{(H)}_s Y_s = \nabla Y_s (\nabla X_s)^{-1}\sigma(s, X_s)    
\end{equation}
for $0 \le s \le T$.

\begin{theorem}
\label{thmpder}
if $Y_t = u(t, X_t)$, where $u(t, x)$ is continuously differentiable with respect to $t$ and twice continuously differentiable with respect with to $x$,
then $u(t, x)$ is the solution of \eqref{eqn10}.
\end{theorem}
\begin{proof}
Apply Theorem \ref{thmito} to $u(t, X_t)$ from $t$ to $T$,
\begin{equation}
\begin{aligned}
u(t+h,x)- u(t,x)=& u(t+h,x)-u(t+h, X_{t+h})+u(t+h, X_{t+h}) -  u(t,x)\\
&=-\int_{t}^{t+h} \mathcal{L}_{frac} u(t+h,X_s) + f \,ds \\ 
&- \sum_{i = 1}^{d} \int_{t}^{t+h}\sigma_i(t+h,X_s)  \frac{\partial u(t+h,X_s)}{\partial x^i}\,dB^H_s\\
&+ \int_{t}^{t+h} Z_s \,dB^H_s \quad a.s.
\end{aligned} 
\end{equation}
where we apply Theorem \ref{thmito} and the fBSDE. Let $t = t_0 < t_1 < \cdots < t_N = T$, then
\begin{equation}
\begin{aligned}
g(X_T) - u(t,x) = &- \sum_{i=0}^{N-1}\int_{t_i}^{t_{i+1}} \left[\mathcal{L}_{frac} u(t_{i+1},X_s) + f(s, X_s, Y_s, Z_s)\right]  \,ds\\
&+ \sum_{i=0}^{N-1} \int_{t_i}^{t_{i+1}}\left[Z_s -\sum_{i = 1}^{d}\sigma_i(t+h,X_s)  \frac{\partial u(t+h,X_s)}{\partial x^i}\right] \,dB^H_s,
\end{aligned}
\end{equation}
in view of Lemma \ref{lemma3.2}, let $\lim_{n \to \infty}\sup_{i \le n-1}(t_i^{n+1}-t_i^{n})=0$, then
\begin{equation}
u(t, x) = g(x) + \int_{t}^{T}\left[\mathcal{L}_{frac} u(s,X_s) + f\left(s, x, u(s, x), \sigma(s, x)\nabla u(s,x)\right) \right]  \,ds.
\end{equation}
\end{proof}

\begin{remark}
From the proof of Theorem \ref{thmpde} and \ref{thmpder}, we can obtain 
\begin{equation}\label{eqn12}
\begin{cases}
\begin{aligned}
&u(t,X_t) = Y_t,\\
&\sigma(t,X_t)\nabla u(t,X_t) = Z_t.
\end{aligned}
\end{cases}
\end{equation}
\end{remark}
And Theorem \ref{thmpde} and \ref{thmpder} tell us the uniqueness and existence of the solution of fBSDE \eqref{eqn3} can be proved by
studying the the uniqueness and existence of the solution of PDE \eqref{eqn10}.
\color{black}
It's worth to consider an example of the geometric fractional Brownian motion which solves the fractional SDE that
\begin{equation}\label{eqn29}
dX_t = \mu X_t dt + \sigma X_t dB_t^H, X_{0} = x_0 \ge 0,
\end{equation}
where $x_0$, $\mu$, $\sigma$ are constants.
\begin{proposition}\label{pro1}
The solution of \eqref{eqn29} is $X_t = x_0\exp^\diamond{\left(\mu t + \sigma B_t^H\right)}$, i.e. $X_t = x_0\exp{\left(\mu t + \sigma B_t^H - \frac{1}{2} \sigma^2 t^{2H}\right)}$.
\end{proposition}
The proof can be found in \cite{HU2011}.

\begin{proposition}\label{pro2}
Let $s \le t$, the solution of \eqref{eqn29} has a derivative $D^\phi_s X_t = \sigma H X_t\left[s^{2H-1} - (t-s)^{2H-1}\right] $ 
\end{proposition}
\begin{proof}
From Lemma \ref{lem1}, $D^\phi_s X_t$ exists and  
\begin{equation}\label{eqn31}
\begin{aligned}
D^\phi_s X_t &= \frac{\partial X_t}{\partial x} D^\phi_s B_t^H\\
&=  \sigma x_0 \exp{\left(\mu t + \sigma B_t^H - \frac{1}{2} \sigma^2 t^{2H}\right)} \int_{0}^{t} \phi(u,s)\, du\\
&= \sigma H X_t \left[s^{2H-1} - (t-s)^{2H-1}\right].
\end{aligned} 
\end{equation}
\end{proof}

\begin{corollary}\label{cor1}
Consider $X_t = \eta_0 + \int_{0}^{t} \mu X_s \,ds + \int_{0}^{t} \sigma X_s \,dB_s^H$, $\mu \in \mathbb{R}$, $\sigma \in \mathbb{R}$, $\eta_0 \in \mathbb{R}^d$ are constants,
if $u \in C^{1,2}([0,T] \times \mathbb{R}^d)$ solves 
\begin{equation}\label{eqn32}
\begin{cases}
\begin{aligned}
\frac{\partial u(t,x)}{\partial t} + \mathcal{L}_{frac} u(t,x) + f(t,x,u(t,x),(\nabla u\sigma)(t,x))&=0,\\
u(T,x)&=g(x),\\   
\end{aligned}
\end{cases}
\end{equation}
and
\begin{equation*}
\mathcal{L}_{frac} = 
\begin{pmatrix}
L_{frac} u_1\\
\cdots\\
L_{frac} u_k
\end{pmatrix},
\end{equation*} 
\begin{equation*}
L_{frac} = \sum_{i,j = 1}^{d}\sigma^2 H x_j^2 t^{2H-1} \frac{\partial^2}{\partial x_i \partial x_j}+\sum_{i = 1}^{d}\mu x_i \frac{\partial}{\partial x_i}.
\end{equation*} 
Then $u(t,x) = Y^{t,x}_t$ where $\left\{(Y_s^{t,x},Z_s^{t,x}), t \le s \le T \right\} $  is the solution of the BSDE \eqref{eqn3}.
\end{corollary}
\begin{proof}
Easily proved from \ref{thmpde} and \ref{pro2}.  
\end{proof}

\section{RNN-BSDE method}\label{sect4}
Before we start to build our network, we should apply a time discretization to BSDEs \eqref{eqn3}. Consider the partition $\pi: 0=t_0 < t_1 < \cdots < t_N = T$, for any $t_n < t _{n+1}$ on $[0,T]$, from Definition \ref{def2.1}, it holds that
\begin{equation}
\label{eqn35}
\begin{cases}
Y_{t_{n+1}} = Y_{t_n} - f(t_n,X_{t_n},Y_{t_n},Z_{t_n})\left(t_{n+1}-t_n\right) + Z_{t_n} \diamond \Delta B^H_{t_n},\\
\end{cases}
\end{equation}
to simulate the fractional Brownian motion $B^H$ at the discrete time, there have already been some methods like the Hosking method\cite{Hosking}, the Cholesky 
method\cite{asmussen1998stochastic}, the Davies and Harte method\cite{davies} and so on. In this paper, we choose the Cholesky 
method. 
\color{black}
Return to our topic, to deal with the Wick product, we use Theorem \ref{thmtrans} and obtain that
\begin{equation}\label{eqn20}
\begin{aligned}
Y_{t_{n+1}} =& Y_{t_n} - f(t_n,X_{t_n},Y_{t_n},Z_{t_n})\left(t_{n+1}-t_n\right) + Z_{t_n}(B^H_{t_{n+1}}-B^H_{t_n})\\
&- D^\phi_{t_n} Z_{t_n}\left(t_{n+1}-t_n\right).
\end{aligned}
\end{equation}
The approximation scheme is still incomplete because $D^\phi_{t_n} Z_{t_n}$ is unknown since $Z_{t_n}$ is which we need to find if it's a problem of solving the fBSDEs.
Of course, it may be an idea that we construct another neural network to approximate $D^\phi_{t_n} Z_{t_n}$ just like what we do for approximating $Z_{t_n}$ that we will introduce soon.
But we prefer to give another way, from Theorem \ref{thmpde} and \eqref{eqn1}, we can consider $Z_t$ as $g(t,X_t)$, so in view of Lemma \ref{lem1} 
\begin{equation*}
\begin{aligned}
D^\phi_s Z_t &= D^\phi_s g(t,X_t)\\ 
&= \frac{\partial g(t,X_t)}{\partial x} D^\phi_s X_t.
\end{aligned}
\end{equation*}
Besides, if $X_t$ is the solution of \eqref{eqn29},
in view of Proposition \ref{pro2},
\begin{equation*}
\begin{aligned}
D^\phi_s Z_t &= \frac{\partial g(t,X_t)}{\partial x} D^\phi_s X_t\\
&= \sigma H X_t\frac{\partial g(t,X_t)}{\partial x}[s^{2H-1}-(s-t)^{2H-1}].
\end{aligned}
\end{equation*}
Then we can rewrite \eqref{eqn20} as
\begin{equation*}
\begin{aligned}
Y_{t_{n+1}} = &Y_{t_n} -\left( f(t_n,X_{t_n},Y_{t_n},Z_{t_n}) + \partial_x^\ast Z_{t_n} \cdot D^\phi_{t_n} X_{t_n}\right)  \Delta t_n\\ 
&+ Z_{t_n}\Delta B_{t_n}^H ,    
\end{aligned}
\end{equation*}
where $\partial_x^\ast$ means automatic differentiation.

\subsection{Main ideas of the RNN-BSDE method}
In our work, we develop an algorithm for solving fractional BSDEs based on deep BSDE method\cite{2017Deep} and refer to it as RNN-BSDE method. The reason why it's necessary to develop a new
algorithm is, as we all know, fractional Brownian motions are not Markov processes. Besides, fBMs have the property of long-range dependence if $H \in (\frac{1}{2},1)$ and short-range dependence if $H \in (0,\frac{1}{2})$(see e.g. \cite{mishura2008stochastic}).
It means the increments of $B^H_t$ can't be non-correlated if $H \ne \frac{1}{2}$. When we approximate $(Z_{t_n},Y_{t_{n+1}})$, it will not be satisfying if we only consider using the 
information of $X_{t_n}$ in the input layer. Instead, we want to make full use of all the information before time $t_{n+1}$, so a recurrent neural network is a better choice than
a feedforward neural network.

Recurrent neural network(RNN) structure\cite{ELMAN1990179}\cite{JORDAN1997471}, which is classical for dealing with time series, has some advantages of solving fractional BSDEs that 
\begin{enumerate}
\item RNN can make full use of more information before time $t_{n+1}$.
RNN processes the sequence data stored in rank-3 tensors of shape (samples, timesteps, features). the computations of $X_{t_n}$, $n=0,1,\cdots,N-1$ in the recurrent unit which is the fundamental building block
of an RNN can be expressed as 
\begin{equation}\label{eqn22}
h_{t_{n}} = f(X_{t_n}*U+h_{t_{n-1}}*W+b),\\    
\end{equation}
where $U$,$W$ are weight matrices, $b$ are biases, $h_{t_n}$ means the hidden state of hidden layers at time $t_n$, $f$ is the activation function. For RNN, consider \eqref{eqn22} with each $t = t_n,\quad n=0,1,\cdots,N-1$, then
\begin{equation*}
\begin{aligned}
h_{t_{n}} =& f(U*X_{t_n}+W*f(U*X_{t_{n-1}}+W*h_{t_{n-2}}+b)+b)\\
=& f(U*X_{t_n}+W*f(U*X_{t_{n-1}}+W*f(\cdots f(U*X_{t_0}+W*h_{0}\\
&+b))+b))\\
=& f(X_{t_n},X_{t_{n-1}},\dots,X_{t_0}),
\end{aligned}
\end{equation*}

clearly, RNN satisfies what we required above.
\item For $\left\{t_n|n = 0,1,2,\cdots,N\right\}  \subseteq [0,T]$, there are $N-1$ FNNs in deep BSDE method, what means the larger N is, the more parameters in neural networks 
that need to be determined are, which may be a burden of compute. But if we use an RNN structure, however large N is, there will be always one RNN since the hidden layer has 
a recurrent structure so that for any timestep $t_n \in \left\{t_n|n = 0,1,2,\cdots,N\right\}$, the weight matrices and biases are common and reused in an epoch. As mentioned in \cite{2017Deep}, for $N+1$ time nodes, one $d$-dimensional input layer, two $(d+10)$-dimensional hidden layers and one $d$-dimensional output layer, 
there will be $\left\{d+1+(N-1) \left[2d(d + 10) + (d + 10)^2 + 4(d + 10)+4d\right]\right\}$ parameters to be trained for deep BSDE while $[1 + 2d(d+10) + 3(d + 10)^2 + d^2 + 6(d+10) + 3d]$ parameters to be trained for RNN-BSDE using a stacked-RNN.
\end{enumerate}

The main idea of the RNN-BSDE method can be expressed as.
\begin{flalign}\label{eqn23}
&\ Z_{t_n} = subRNN(X_{t_n};\theta),  \quad n=0,1,\cdots,N-1&
\end{flalign}
\begin{flalign}\label{eqn24}
&\
\begin{aligned}
Y_{t_{n+1}} = &Y_{t_n} -\left( f(t_n,X_{t_n},Y_{t_n},Z_{t_n}) + \partial_x^\ast Z_{t_n} \cdot D^\phi_{t_n} X_{t_n}\right)  \Delta t_n\\ 
&+ Z_{t_n}\Delta B_{t_n}^H ,    
\end{aligned}
&,
\end{flalign}
\begin{flalign}\label{eqn25}
&\ loss = E\left\lvert g(X_T)-Y_T\right\rvert ^2,&
\end{flalign}
\begin{flalign}\label{eqn26}
&(Y_0^{new}, Z_0^{new}, \left\{Z_n^{new}\right\}, \theta^{new} ) = BP(loss(DNN(X_t))).  \quad n=1,\cdots,N-1 &
\end{flalign}     
In \eqref{eqn23}, the sub-neural network is an RNN instead of $N-1$ FNNs, and to make the RNN more effective, we choose a stacked RNN rather than a simple RNN, which the structure is shown in Fig.\ref{fig1}. 
\begin{figure}[h]
\centering
\includegraphics{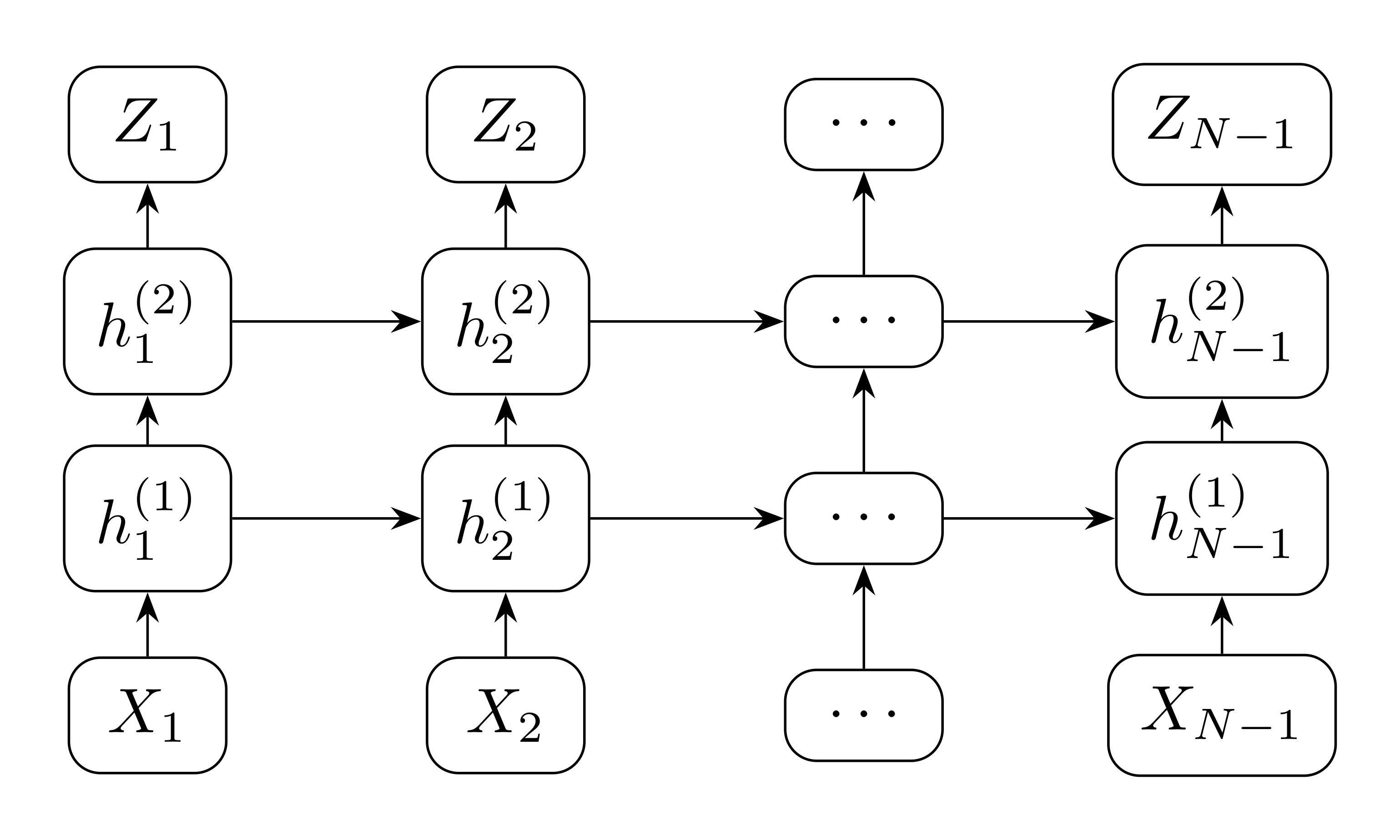}
\caption{Rough sketch of the architecture of the stacked RNN. $t_n$ is simply written as $n$, $h_{j}^{(i)}$ means the hidden state of the $i$th hidden layer at time $t_j$.\label{fig1}}
\end{figure}

The whole flow in the direction of forward propagation can be seen in Fig.\ref{fig2}.
Besides, when apply the deep BSDE method, batch normalization\cite{bn} is adopted right after each matrix multiplication and before activation. Notice that batch normalization is not suitable for RNNs, instead, we choose layer normalization\cite{ba2016layer}.
Finally, we provide the pseudocode of RNN-BSDE method as following:   
\begin{figure}
\centering
\includegraphics{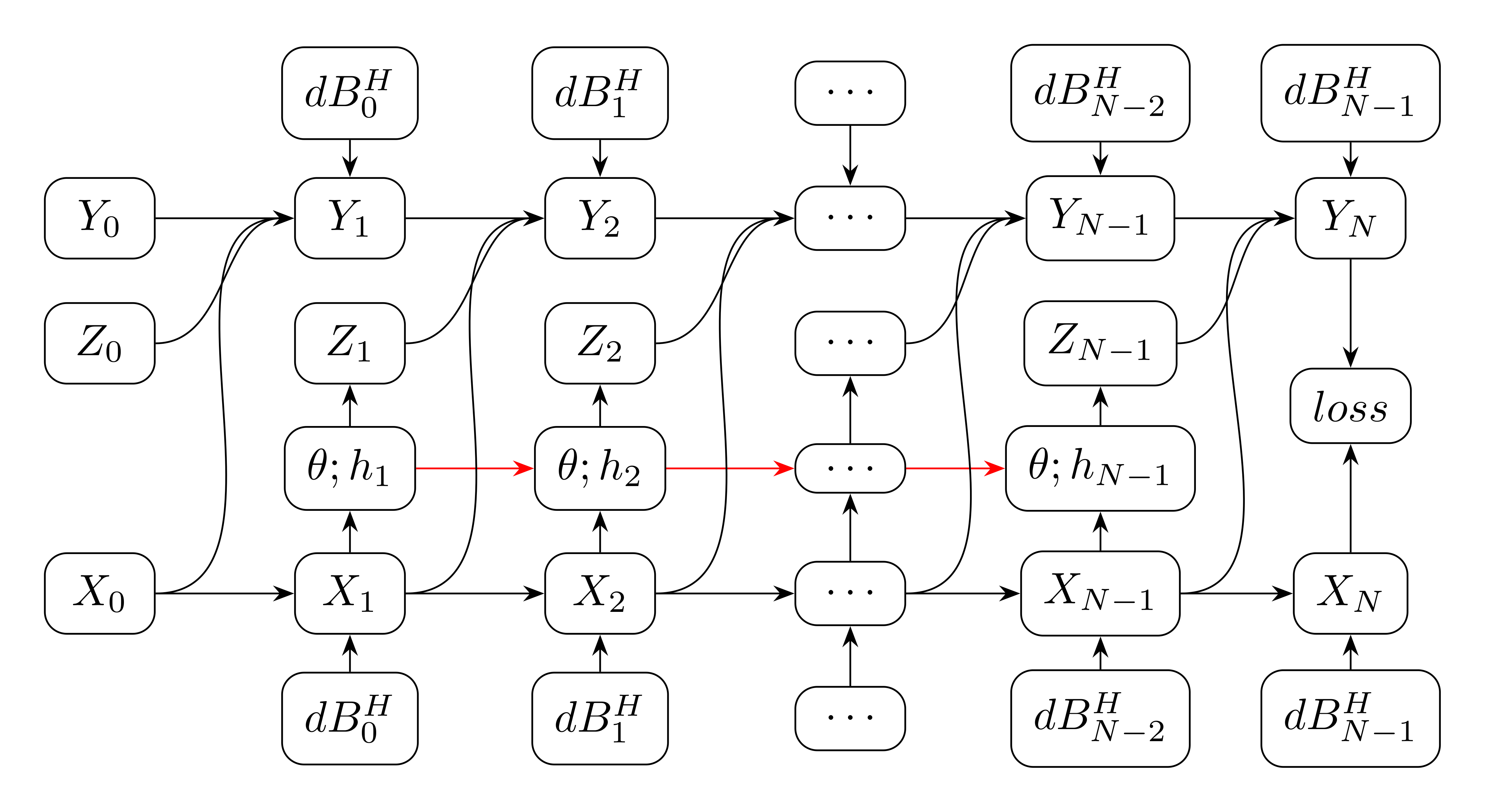}
\caption{Rough sketch of the architecture of the RNN-BSDE method. $\theta$ represents the parameters to be trained.\label{fig2}}
\end{figure}                   
\begin{algorithm}
    \caption{The RNN-BSDE method}
    \begin{algorithmic}[1]
    \Require Initial parameters $\left(y^{m, 0}_0, \theta^0 \right)$, $\left\{X^m_{t_j}\right\}$ samples;
    \Ensure $\left(\tilde{Y}^{m, i}, \tilde{Z}^{m, i}\right)$;
    \For{$i=0 \to maxstep-1$}
    \State $ \left\{\tilde{Z}^{m, i}_{t_j} \right\} \gets sub\mathcal{RNN}(\left\{X^m_{t_j}\right\};\theta^i)$;
    \State $\left\{\partial_x^\ast \tilde{Z}^{m, i}_{t_j}\right\} \gets AD\left(\left\{\tilde{Z}^{m, i}_{t_j}\right\}, \left\{X^m_{t_j}\right\} \right)$; \Comment{Automatic differentiation}
    \For{$j =0 \to N-1$}
    \State $\tilde{Y}^{m,i}_{t_{j+1}} \gets \tilde{Y}^{m, i}_{t_j} - \left( f(t_j,X^m_{t_j},\tilde{Y}^{m, i}_{t_j},\tilde{Z}^{m, i}_{t_j}) + \partial_x^\ast \tilde{Z}^{m, i}_{t_j} \cdot D^\phi_{t_j} X^m_{t_j}\right) \Delta t+ \tilde{Z}^{m, i}_{t_j}\Delta B_{t_j}^H$;
    \EndFor
    \State Set loss function $\frac{1}{M} \sum_{m = 1}^{M}  \left\lvert g(X^m_T)-\tilde{Y}^{m, i}_T\right\rvert^2$;
    \State $\left(y^{m, i+1}_0, \theta^{i+1}\right) \gets Adam(loss;y^{m, i}_0, \theta^i)$;
    \EndFor
    \end{algorithmic}
\end{algorithm}

\subsection{More detail of the RNN-BSDE method and an example of solving fBSDEs}
In this section, we will introduce more about how to set up the neural network of the RNN-BSDE method to make the algorithm more practical. It's convenient to use 
the simplest traditional RNN to explain our main idea but it's too weak to solve some fBSDEs. So it's necessary to tell more detail and apply some more practical types of RNN.

Suppose the samples of $X_t$, $\widetilde{X_t} \in \mathbb{R}^{m \times (N+1) \times d}$, where $m \in \{m_1, m_2\}$, $m_1$ means the number of 
sample paths in the whole valid sets and we denote the mini-batch size by $m_2$, $N + 1$ means the number of time nodes, $d$ is the dimension
(which will be regarded as the number of features in deep learning) of inputs.
The stacked RNN to approximate $(Z_{t_n}, n = 0,1,\cdots,N)$ has at least four layers including one $d$-dimensional input layer, 
at least two $(d+10)$-dimensional hidden layers and one $d$-dimensional output layer. The hidden layers and output layer are all RNN layers composed of recurrent units, we set weight matrices 
$U_h \in \mathbb{R}^{d \times \left(d+10\right)}$ and $U_o \in \mathbb{R}^{d \times d}$, recurrent weight matrices $W_h \in \mathbb{R}^{\left(d+10\right) \times \left(d+10\right)}$ and $W_o \in \mathbb{R}^{\left(d+10\right) \times d}$
. There is no activation function after each matrix multiplication and instead we have
\begin{equation}\label{eqn27}
\begin{aligned}
&h^{(1)}_{t_n} = tanh\left(LN^{(1)}\left(X_{t_n}*U_h^{(1)}+ h_{t_{n-1}}^{(1)}*W_h^{(1)}\right) + b^{(1)} \right),\\ 
&h^{(2)}_{t_n} = tanh\left(LN^{(2)}\left(h_{t_n}^{(1)}*U_h^{(2)}+ h_{t_{n-1}}^{(2)}*W_h^{(2)}\right) + b^{(2)} \right),\\
&Z_{t_{n}} = LN^{(3)}\left(h^{(2)}_{t_{n}}*U_o+ Z_{t_{n-1}}*W_o\right) + b^{(3)} ,\quad n=0,1,\cdots,N-1
\end{aligned}  
\end{equation}
where $LN$ means layer normalization, $tanh$ is the hyperbolic tangent function. \eqref{eqn27} can be understood more easily with Fig.\ref{fig1} together.
All weights and parameters of layer normalization will be randomly initialised at the start of each run. 

If we worry that a stacked RNN is still not powerful 
enough to solve most fBSDEs, it will be the turn of a special type of RNN named Long Short Term Memory networks(LSTMs)\cite{LSTM}. For $H \in \left(\frac{1}{2}, 1 \right) $, the fBM $B_t^H$ holds the long memory property and as known a traditional RNN is not able to 
handle "long-term dependencies" in practice while the LSTM can keep useful with long-term dependencies and deal with the vanishing gradient problem in the RNN. Since LSTMs are a kind of
RNN, we can change an RNN into a LSTM just by replacing RNN units in the network with LSTM units which are the fundamental building blocks of a LSTM. A LSTM cell is composed of
a cell and three gates including an input gate, a forget gate and an output gate. In the RNN-BSDE method, we choose to use a LSTM with layer normalization which has the similar structure to 
a stacked-RNN illustrated in Fig.\ref{fig1}, i.e. multiple $(d+10)$-dimensional LSTM layers as hidden layers and one extra $d$-dimensional LSTM layer before output.

\subsection{Convergence analysis}
In this part, we provide a posterior estimate of the numerical solution and this posterior estimate justifies the convergence of RNN-BSDE method. Firstly, assume 
\begin{assumptions}
\label{asm1}
Let $u \in C^{1,2}(\mathbb{R}^+ \times \mathbb{R}^d)$, $u$ and the process $X_t$ satisfy the conditions in \ref{thmito} to 
ensure that Itô formula for fWIS integrals(Theorem.\ref{thmito}) is applicable to $u(t, X_t)$. 
\end{assumptions}
\begin{assumptions}
\label{asm2}
For any $y, y^{\prime}$,$x, x^{\prime}$ and $t, t^{\prime}$, 
\begin{equation*}
\left\lvert f(t,x,y,z) - f(t^{\prime},x^{\prime},y^{\prime},z)\right\rvert \le L_f \left(\left\lvert t-t^{\prime} \right\rvert + \left\lvert x-x^{\prime} \right\rvert + \left\lvert y-y^{\prime} \right\rvert\right),
\end{equation*}
where $L_f$ is a given positive constant.
\end{assumptions}
\begin{assumptions}
\label{asm3}
For any $z, z^{\prime} $ and any $t_i$, $t_j \in [0,T]$ satisfying $t_i \le t_j$,
\begin{equation*}
\int_{t_i}^{t_j}  \left\lvert f(t,x,y,z) - f(t,x,y,z^{\prime})\right\rvert^2  \,dx \le C_f E \left\lvert \int_{t_i}^{t_j} z-z^{\prime} \,dB_t^H\right\rvert^2,
\end{equation*}
where $C_f$ is a given positive constant and denote $C_f^\ast = C_f \vee L_f$.
\end{assumptions}
\begin{assumptions}
\label{asm5}
Assume \eqref{eqn10} has a unique classical solution, then there exist a unique pair of $\mathscr{F}_t^H$-adapted processes $\left(Y_t,Z_t\right)$ which solve \eqref{eqn3}. 
\end{assumptions}

Consider the following fFBSDE system with the state $\widetilde{Y}_t$  
\begin{equation}
\label{eqn39}
\begin{cases}
\begin{aligned}
&X_t = \xi + \int_{0}^{t} \mu(s,X_s) \,ds  + \int_{0}^{t} \sigma (s,X_s)\, d B_s^H,\\
&\widetilde{Y}_t = \widetilde{y} - \int_{0}^{t} f(s,X_s,\widetilde{Y}_s,\widetilde{Z}_s) ds + \int_{0}^{t} \widetilde{Z}_s \,d B_s^H,\\      
\end{aligned}
\end{cases}
\end{equation}
the aim is to minimize the objective functional
\begin{equation}
\label{eqn40}
J \left(\widetilde{y},\widetilde{Z}_\centerdot\right) = E\left\lvert g(X_T)-\widetilde{Y}_T\right\rvert ^2,    
\end{equation}
under the control $\left(\widetilde{y},\widetilde{Z}_\centerdot\right) \in \mathbb{R} \times \mathcal{L}(0,T)$.

Firstly, notice \eqref{eqn38}, and we provide 
\begin{theorem}
\label{thmct}
Assume Assumptions \ref{asm1}-\ref{asm5} hold. Let $(Y_t, Z_t)$ be the solution of fBSDE \eqref{eqn3}, $\widetilde{Y}_t$ denotes the state of 
system \eqref{eqn39} under the control $\left(\widetilde{y},\widetilde{Z}_\centerdot\right) \in \mathbb{R} \times \mathcal{L}(0,T)$, then there 
exists some constant $C$ which only depends on $T$ and $C_f^\ast$ satisfying
\begin{equation}
\label{eqn41}
\begin{aligned}
&\sup_{0 \le t \le T}E \left\lvert Y_t - \widetilde{Y}_t\right\rvert^2 + E \left\lvert \int_{0}^{T} Z_s - \widetilde{Z}_s \,dB_s^H\right\rvert^2\\
=&\sup_{0 \le t \le T}E \left\lvert Y_t - \widetilde{Y}_t\right\rvert^2 + E \left\{ \left(\int_{0}^{T} D^\phi_s \left(Z_s - \widetilde{Z}_s\right)  \,ds\right) ^2 + \left\lvert Z_s - \widetilde{Z}_s \right\rvert _\phi \right\} \\
\le&C E\left\lvert g(X_T)-\widetilde{Y}_T\right\rvert ^2 .
\end{aligned}
\end{equation}
\end{theorem}
The proof of Theorem \ref{thmct} is an analogue of Theorem 2.1 in \cite{Jiang_2021}, only note that Itô formula should be replaced by 
the Itô formula for Wick integration.

In view of \ref{thmct}, the problem of solving \eqref{eqn3} can be changed into the stochastic control problem of the system \eqref{eqn39}. So it will be
a reasonable choice to apply deep learning to this kind of problems.

Then, we need to provide the estimation of the error resulted from time discretization. From now on, we manily consider $d = 1$ for brevity. 

$C^{\gamma}([0,T])$ denotes a $\gamma$-H$\mathrm{\ddot{o}}$lder space, $\left\lVert \cdot \right\rVert _\gamma$ denotes the $\gamma$-H$\mathrm{\ddot{o}}$lder norm.

For any constant $K > 0$, define 
\begin{equation*}
\mathcal{A}_{K} := \left\{Z \in \mathcal{L}(0,T)| Z(\cdot, \omega) \in C^{\frac{1}{2}}([0,T]), E\left(\left\lVert Z \right\rVert _{\frac{1}{2}}^4\right)   \le K \right\}.
\end{equation*}

Let $\widehat{Y}_t$ be the state of 
\begin{equation}
\label{eqn43}
\begin{cases}
\begin{aligned}
&\widehat{X}_{t_{k+1}} = \widehat{X}_{t_k} + \mu(t_k,\widehat{X}_{t_k}) \Delta t  + \sigma (t_k,\widehat{X}_{t_k})\,\diamond \Delta B^H_{t_k},\\
&\widehat{Y}_{t_{k+1}} = \widehat{Y}_{t_k} - f(t_k,\widehat{X}_{t_k},\widehat{Y}_{t_k},\widetilde{Z}_{t_k})\left(t_{k+1}-t_k\right) + \widetilde{Z}_{t_k} \diamond \Delta B^H_{t_k},
\end{aligned}
\end{cases}
\end{equation}
with the aim of minimizing the objective functional
\begin{equation}
\label{eqn44}
\widehat{J} \left(\widetilde{y},\widetilde{Z}_\centerdot\right) = E\left\lvert g(X_T)-\widehat{Y}_T\right\rvert ^2,    
\end{equation}
under the control $\left(\widetilde{y},\widetilde{Z}_\centerdot\right) \in \mathbb{R} \times \mathcal{A}_{K}$. For any partition $\pi$, define

\begin{lemma}
\label{lemfinal}
Assume $\widetilde{Z} \in \mathcal{A}_{K}$, then let $N$ large enough, for any partition $\pi$ and $k = 0,\dots,N-1$, it follows that
\begin{equation}
E \left\lvert \int_{t_k}^{t_{k+1}} \widetilde{Z}_t \,dB_t^H - \widetilde{Z}_{t_k} \diamond \Delta B^H_{t_k} \right\rvert^2 \le C_{K} \left[\sqrt{E \left(\left\lVert B^H \right\rVert^4 _{C^{H-\delta}}\right)} \left(\Delta t\right)^{1 + 2H-2\delta}+\Delta t^2  \right] ,
\end{equation}
moreover, 
\begin{equation}
\sup_{t_k \le t \le t_{k+1}} E\left\lvert \widetilde{Y}_t-\widehat{Y}_{t_k} \right\rvert^2 \le C_{K} \left[\sqrt{E \left(\left\lVert B^H \right\rVert^4 _{C^{H-\delta}}\right)} \left(\Delta t\right)^{1 + 2H-2\delta}+\Delta t^2  \right],
\end{equation}
especially, 
\begin{equation}
E\left\lvert \widetilde{Y}_T-\widehat{Y}_{T} \right\rvert^2 \le C_{K} \left[\sqrt{E \left(\left\lVert B^H \right\rVert^4 _{C^{H-\delta}}\right)} \left(\Delta t\right)^{1 + 2H-2\delta}+\Delta t^2  \right],
\end{equation}
with some $C_{K}$ not depending on $B^H$, $\delta \in \left(0, H-\frac{1}{2}\right) $.
\end{lemma}
\begin{proof}
Define $Z_t^\ast := \widetilde{Z}_t- \widetilde{Z}_{t_k}$ for $t \in \left[t_k, t_{k+1}\right] $. In view of Lemma 19 in \cite{1999On}, it follows that
\begin{equation}
\left\lvert \int_{t_k}^{t_{k+1}} Z_t^\ast \,\delta B_t^H \right\rvert \le C \left\lVert Z_t^\ast \right\rVert _{\frac{1}{2}} \left\lVert B^H \right\rVert _{C^{H-\delta}}\left(\Delta t\right)^{\frac{1}{2}+H-\delta}
\end{equation}
where $\frac{1}{2}+H-\delta > 1$. Then
\begin{equation}
E \left\lvert \int_{t_k}^{t_{k+1}} \widetilde{Z}_t \,\delta B_t^H - \widetilde{Z}_{t_k} \diamond \Delta B^H_{t_k} \right\rvert^2 \le C_{K} \left[\sqrt{E \left(\left\lVert B^H \right\rVert^4 _{C^{H-\delta}}\right)} \left(\Delta t\right)^{1 + 2H-2\delta}\right].
\end{equation}
In view of Theorem \ref{thmtrans} and \eqref{eqn38}, for $N$ large enough,
\begin{equation}
\begin{aligned}
E \left\lvert \int_{t_k}^{t_{k+1}} Z_t^\ast \,d B_t^H \right\rvert^2 &\le E \left\lvert \int_{t_k}^{t_{k+1}} Z_t^\ast \,\delta B_t^H \right\rvert ^2 + E\left\{\left(\int_{t_k}^{t_{k+1}} D^\phi_s Z_s^\ast\,ds\right)^2\right\}\\
&\le C_{K} \left[\sqrt{E \left(\left\lVert B^H \right\rVert^4 _{C^{H-\delta}}\right)} \left(\Delta t\right)^{1 + 2H-2\delta}\right] \\
&+ \sup_{t_k \le t \le t_{k+1}}E\left\lvert D^\phi_t \left(Z_t - Z_{t_k}\right) \right\rvert^2 \Delta t^2\\
&\le C_{K} \left[\sqrt{E \left(\left\lVert B^H \right\rVert^4 _{C^{H-\delta}}\right)} \left(\Delta t\right)^{1 + 2H-2\delta}+\Delta t^2  \right].     
\end{aligned}
\end{equation}

let $b_0$ denote $C_{K} \left[\sqrt{E \left(\left\lVert B^H \right\rVert^4 _{C^{H-\delta}}\right)} \left(\Delta t\right)^{1 + 2H-2\delta}+\Delta t^2  \right]$.
Consider \eqref{eqn43} and \eqref{eqn39}, for any $t \in [t_{k+1}, t_{k+2}]$ we have
\begin{equation}
\label{eqn45}
\begin{aligned}
E\left\lvert \widetilde{Y}_t - \widehat{Y}_{t_{k+1}} \right\rvert^2 \le& \left(1 + \Delta t\right) E\left\lvert \widetilde{Y}_{t_k} - \widehat{Y}_{t_k} \right\rvert^2  \\
+&\left( \Delta t + \Delta t^2  \right) \int_{t_k}^{t_{k+1}} E \left\lvert f(s,X_s,\widetilde{Y}_s,\widetilde{Z}_s) -  f(t_k,\widehat{X}_{t_k},\widehat{Y}_{t_k},\widetilde{Z}_{t_k})\right\rvert^2 \,ds \\
+&E \left\lvert \int_{t_k}^{t_{k+1}} \widetilde{Z}_t \,dB_t^H - \widetilde{Z}_{t_k} \diamond \Delta B^H_{t_k} \right\rvert^2\\
\le& \left(1 + \Delta t\right) E\left\lvert \widetilde{Y}_{t_k} - \widehat{Y}_{t_k} \right\rvert^2 \\
+& C\Delta t \int_{t_k}^{t_{k+1}}E\left(\left\lvert X_s-\widehat{X}_{t_k} \right\rvert^2 + \left\lvert \widetilde{Y}_s-\widehat{Y}_{t_k} \right\rvert^2\right)  \,ds\\
+& \left(1 + C \Delta t \right)  E \left\lvert \int_{t_k}^{t_{k+1}} \widetilde{Z}_t \,dB_t^H - \widetilde{Z}_{t_k} \diamond \Delta B^H_{t_k} \right\rvert^2\\
\le& \left(1+ C\Delta t\right) \sup_{t_k \le t \le t_{k+1}}E\left\lvert \widetilde{Y}_{t} - \widehat{Y}_{t_k}\right\rvert^2 \\
+& C\Delta t \int_{t_k}^{t_{k+1}}E\left(\left\lvert X_s-\widehat{X}_{t_k} \right\rvert^2 \right)  \,ds+ \left(1+ C\Delta t\right)b_0,
\end{aligned}
\end{equation}
since we let $N$ large enough, and apply supremum to \eqref{eqn45} 
\begin{equation}
\sup_{t_{k+1} \le t \le t_{k+2}} E\left\lvert \widetilde{Y}_t - \widehat{Y}_{t_{k+1}} \right\rvert^2 \le \left(1+ C\Delta t\right)\left(\sup_{t_k \le t \le t_{k+1}} E\left\lvert \widetilde{Y}_t - \widehat{Y}_{t_{k}} \right\rvert^2 + b_0 \right),
\end{equation}
especially, we have
\begin{equation}
E\left\lvert \widetilde{Y}_T - \widehat{Y}_{T} \right\rvert^2 \le \left(1+ C\Delta t\right)\left(\sup_{t_{N-1} \le t \le T} E\left\lvert \widetilde{Y}_{t} - \widehat{Y}_{t_{N-1}} \right\rvert^2 + b_0 \right),    
\end{equation}

Define $a_k := \sup_{t_{k-1} \le t \le t_{k}} E\left\lvert \widetilde{Y}_t - \widehat{Y}_{t_{k-1}} \right\rvert^2$, $k = 1,\dots,N$.
Notice $\widetilde{Y}_0 = \widehat{Y}_{t_{0}} = \widetilde{y}$, instantly we have $a_1 = (1 + C \Delta t)b_0$, then
\begin{equation}
\label{eqn46}
\begin{cases}
\begin{aligned}
&a_k \le \left(1+C\Delta t\right) \left(a_{k-1} + b_0\right),\\
&a_1 \le (1 + C \Delta t)b_0,
\end{aligned}    
\end{cases}    
\end{equation}
in view of \eqref{eqn46}, we have
\begin{equation}
\label{eqn47}
a_N \le \sum_{k = 1}^{N} \left(1+C\Delta t\right)^{k} b_0 \le C_{K} \left[\sqrt{E \left(\left\lVert B^H \right\rVert^4 _{C^{H-\delta}}\right)} \left(\Delta t\right)^{1 + 2H-2\delta}+\Delta t^2  \right],
\end{equation}
and $a_k$, $k = 1,\dots,N-1$, $E\left\lvert \widetilde{Y}_T - \widehat{Y}_{T} \right\rvert^2$ also hold, the proof is finished.
\end{proof}

Finally, we can give
\begin{theorem}
\label{thmtd}
Assume Assumptions \ref{asm1}-\ref{asm5} hold. Let $(Y_t, Z_t)$ be the solution of fBSDE \eqref{eqn3}, $\widehat{Y}_t$ denotes the state of 
system \eqref{eqn43} under the control $\left(\widetilde{y},\widetilde{Z}_\centerdot\right) \in \mathbb{R} \times \mathcal{A}_{K}$, then for
$N$ large enough, there exist some constants $C$ which only depends on $T$ and $C_f^\ast$ and $C_K$ which depends on $T$, $C_f^\ast$ and $K$ satisfying
\begin{equation}
\label{eqn48}
\begin{aligned}
&\max_{0 \le k \le N}\sup_{t_{k-1} \le t \le t_{k}}E \left\lvert Y_t - \widehat{Y}_{t_{k-1}} \right\rvert^2 + E \left\lvert \int_{0}^{T} Z_s - \widetilde{Z}_s \,dB_s^H\right\rvert^2\\
=&\max_{0 \le k \le N}\sup_{t_{k-1} \le t \le t_{k}}E \left\lvert Y_t - \widehat{Y}_{t_{k-1}} \right\rvert^2 + E \left\{ \left(\int_{0}^{T} D^\phi_s \left(Z_s - \widetilde{Z}_s\right)  \,ds\right) ^2 + \left\lvert Z_s - \widetilde{Z}_s \right\rvert _\phi \right\} \\
\le&C E\left\lvert g(X_T)-\widehat{Y}_{T}\right\rvert ^2 + C_{K} \left[\sqrt{E \left(\left\lVert B^H \right\rVert^4 _{C^{H-\delta}}\right)} \left(\Delta t\right)^{1 + 2H-2\delta}+\Delta t^2  \right],
\end{aligned}
\end{equation}
where $\delta \in \left(0, H-\frac{1}{2}\right)$.
\end{theorem}
\begin{proof}
In view of Theorem \ref{thmct} and Lemma \ref{lemfinal},
\begin{equation}
\begin{aligned}
&\max_{0 \le k \le N}\sup_{t_{k-1} \le t \le t_{k}}E \left\lvert Y_t - \widehat{Y}_{t_{k-1}} \right\rvert^2 + E \left\lvert \int_{0}^{T} Z_s - \widetilde{Z}_s \,dB_s^H\right\rvert^2\\
\le&\sup_{0 \le t \le T}E \left\lvert Y_t - \widetilde{Y}_t\right\rvert^2 + \max_{0 \le k \le N}\sup_{t_{k-1} \le t \le t_{k}}E \left\lvert \widetilde{Y}_t - \widehat{Y}_{t_{k-1}} \right\rvert^2+ E \left\lvert \int_{0}^{T} Z_s - \widetilde{Z}_s \,dB_s^H\right\rvert^2\\
\le&C E\left\lvert g(X_T)-\widetilde{Y}_T\right\rvert ^2  + \max_{0 \le k \le N}\sup_{t_{k-1} \le t \le t_{k}}E \left\lvert \widetilde{Y}_t - \widehat{Y}_{t_{k-1}} \right\rvert^2\\
\le& C E\left\lvert g(X_T)-\widehat{Y}_T\right\rvert ^2 + C E\left\lvert \widetilde{Y}_T-\widehat{Y}_T\right\rvert ^2 + \max_{0 \le k \le N}\sup_{t_{k-1} \le t \le t_{k}}E \left\lvert \widetilde{Y}_t - \widehat{Y}_{t_{k-1}} \right\rvert^2\\
\le&C E\left\lvert g(X_T)-\widehat{Y}_{T}\right\rvert ^2 + C_{K} \left[\sqrt{E \left(\left\lVert B^H \right\rVert^4 _{C^{H-\delta}}\right)} \left(\Delta t\right)^{1 + 2H-2\delta}+\Delta t^2  \right].
\end{aligned}    
\end{equation}
\end{proof}


\section{Numerical examples}
\label{sect5}
In this section, We will introduce some experiments to verify whether RNN-BSDE method works well on fractional BSDEs. We mainly apply the RNN-BSDE algorithm based on a multi-layer LSTM to our experiments, which we refer to as LSTM-BSDE for brevity. 
In addition, we refer to the RNN-BSDE algorithm based on a stacked RNN as mRNN-BSDE.
\subsection{Fractional Black-Scholes equation}
In this subsection, we consider the extension of the famous Black-Scholes equation\cite{1973bs} which is widely applied in the field of finance.
In view of \eqref{cor1}, the fractional Black-scholes equation in the case of $d=1$ has the form of 
\begin{equation}
\label{eqn33}
\begin{cases}
\begin{aligned}
&\frac{\partial u(t,x)}{\partial t} + \sigma^2 H x^2 t^{2H-1} \frac{\partial^2 u(t,x)}{\partial x^2}+ r x \frac{\partial u(t,x)}{\partial x} - ru(t,x) =0,\\
&u(T,x)=g(x),\\   
\end{aligned}
\end{cases}
\end{equation}
where $r$ is a constant known as the interest rate. And if $H = \frac{1}{2}$, \eqref{eqn33} will be exactly the famous standard Black-Scholes equation.

Adopt $g(x) = \max{\left\{x-K, 0 \right\}}$, To solve \eqref{eqn33} is the equivalent of solving the pricing problem of European call option.
It is not a difficult thing and just similar to what to do to solve the standard Black-Scholes equation in the $1-d$ case. By means of variable substitution, 
change \eqref{eqn33} into a typical heat equation, and it can be verified that the solution of \eqref{eqn33} is 
\begin{equation}
\label{eqn34}
u(t,x) = xN(d_1) - Ke^{-r(T-t)}N(d_2),
\end{equation}
where $N(t) = \frac{1}{\sqrt{2\pi}}\int_{-\infty}^{t}e^{-\frac{s^2}{2}} \,ds $ is the normal distribution function and
\begin{equation*}
\begin{aligned}
\eta &= \frac{\ln{\frac{x}{K}} + r(T-t)}{\sigma \sqrt{T^{2H}-t^{2H}}},\\
d_1 &= \eta + \frac{\sigma}{2}\sqrt{T^{2H}-t^{2H}},\\
d_2 &= \eta - \frac{\sigma}{2}\sqrt{T^{2H}-t^{2H}}.\\
\end{aligned}
\end{equation*}

Our goal is to approximate $u(0,x_0)$, $x_0 = (x_0^{(1)},\dots, x_0^{d}) \in \mathbb{R}^d$ by the deep learning method and compare the LSTM-BSDE method with other methods
designed for solving high-dimensional PDEs and SDEs to verify whether RNN-BSDE method works well on fractional BSDEs. There is some common setting for LSTM-BSDE method.
The multi-layer LSTM set in the LSTM-BSDE consists of one input layer, two hidden layers and one output layer. The input layer is $d$-dimensional , the two hidden layers
are $(d + 10)$-dimensional, the output layer is $d$-dimensional. In every hidden layer and output layer, Xavier initialisation\cite{2010Understanding} is used to initialise weights of inputs, 
orthogonal initialisation is used to initialise weights of recurrent connections, the biases are initialised to zero (these are exactly the default setting in Keras for LSTM units).
The normal initialisation and the uniform initialisation are used to initialise $\beta$ and $\gamma$ of layer normalization. The layer normalization is applied before all the activation functions in
the LSTM units of all hidden layers, and before the output layer.

The setting for the stacked RNN used in the RNN-BSDE method is the same as what we set for LSTM-BSDE. The methods used to compare with LSTM-BSDE that we choose are deep splitting method\cite{Beck_2021} and DBDP1 method\cite{Cme2020Deep}. The neural networks of these methods are 
FNN-based if without any extra description. And we set these FNNs as same as the one for the Deep BSDE method described in \cite{2017Deep}.

As for the optimizer, we choose Adam\cite{2014Adam} for all, which is effective as known and checked by our experiment results.
\subsubsection{Results in the one-dimension case($d=1$)}
Set the dimension $d=1$, $0 = t_0 < t_1 < \dots < t_N = T$, $T = 0.5$, $N = 20$, then $\Delta t = T/N = 0.025$. For the parameters of \eqref{eqn29}, \eqref{eqn33} and \eqref{eqn34}, 
$\mu = 0.06$, $\sigma = 0.2$, $x_0 = 100$, $r = 0.06$, $K = 100$. Learning rate $lr = 0.005$, the valid set size $m_1 = 256$ and the mini-batch size $m_2 = 64$. To approximate $u(0,x_0)$,
there will be 5 independent runs for each of the methods.

For comparison, firstly, we consider a trival case, where $H = \frac{1}{2}$, i.e. \eqref{eqn29} is a standard SDE driven by the Brownian motion $B_t$ and 
the explicit solution $u(0, x_0)$ in \eqref{eqn34} is around 7.1559. It is not surprising to observe that in Fig.\ref{fig3} the results of $\widetilde{u}(0, x_0)$ from all algorithms are close to the true value 
and these methods using FNNs except deep splitting have a little better performance than these using RNNs since the Brownian motion $B_t$ has such a well-known fine property named Markov property.
It means that to forecast the information in the future, we only need to know the information at the moment without considering what happened in the past, 
which makes the RNN structure lose its advantage. 
\begin{figure}
\centering
\includegraphics[width=.9\textwidth]{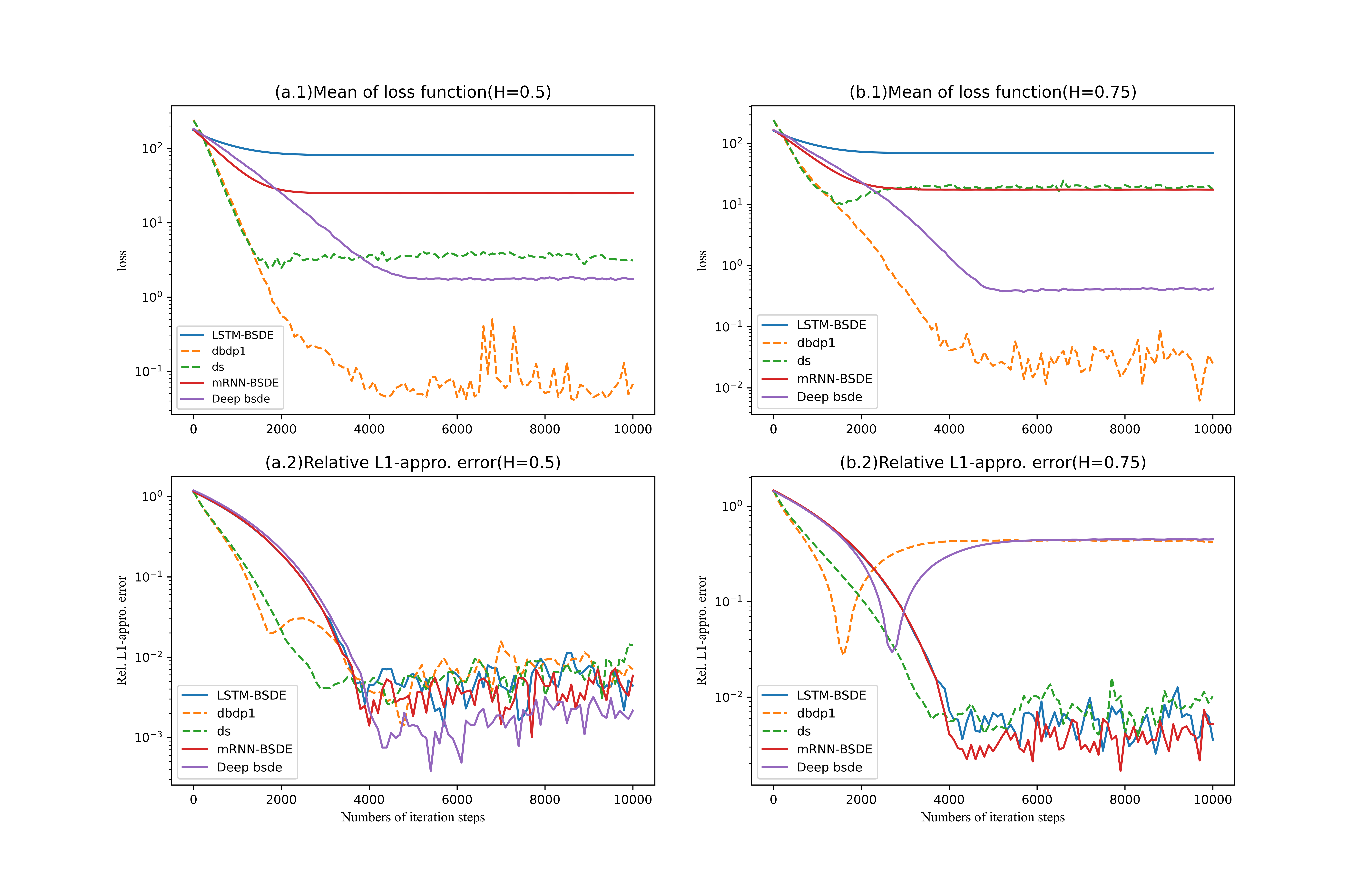}
\caption{Mean of the loss function and relative $L^1$-approximation error of $u(0,x_0)$ in the $1-d$ case of the PDE \eqref{eqn33}.
(a.1) mean of the loss function when $H = 1/2$; (a.2)relative $L^1$-approximation error when $H = 1/2$;
(b.1) mean of the loss function when $H = 3/4$; (b.2)relative $L^1$-approximation error when $H = 3/4$.\label{fig3}}
\end{figure}

\begin{table}
\caption{Numerical simulations for each method in the $1-d$ case of the PDE \eqref{eqn33} with $H = \frac{1}{2}(u_{true} = 7.1559)$.\label{tab1}}
\begin{tabular}{lccclc}
\hline
                            & \begin{tabular}[c]{@{}c@{}}Means of\\ $u_0$\end{tabular} & \begin{tabular}[c]{@{}c@{}}Std of\\ $u_0$\end{tabular} & \begin{tabular}[c]{@{}c@{}}Rel.$L^1$\\ error\end{tabular} & \begin{tabular}[c]{@{}l@{}}Std of\\ rel.error\end{tabular} & \begin{tabular}[c]{@{}c@{}}Avg.runtime\\ /s\end{tabular} \\ \hline
deep BSDE                 & 7.1565                                                   & 5.82e-3                                                & 5.88e-4                                                   & 5.11e-4                                                    & 266                                                      \\
LSTM-BSDE                 & 7.1458                                                   & 1.98e-2                                                & 2.43e-3                                                   & 1.77e-3                                                    & 394                                                      \\
mRNN-BSDE                 & 7.1502                                                   & 1.84e-2                                                & 2.20e-3                                                   & 1.40e-3                                                    & 174                                                      \\
\multicolumn{1}{c}{DS}    & 7.1400                                                   & 3.01e-2                                                & 3.74e-3                                                   & 2.72e-3                                                    & 795                                                      \\
\multicolumn{1}{c}{DBDP1} & 7.1586                                                   & 9.54e-3                                                & 9.63e-4                                                   & 8.15e-4                                                    & 636                                                      \\ \hline
\end{tabular}
\end{table}
\newpage
Focus on the numerical experiments with $H = \frac{3}{4}$, the explicit solution $u(0, x_0)$ in \eqref{eqn34} is around 6.2968, in this case,
the LSTM-BSDE method and mRNN-BSDE begin to make a difference, $\widetilde{u}(0, x_0)$ by the LSTM-BSDE method is close to the true value $u(0, x_0)$ while the deep BSDE method and 
DBDP1 both offer the results of $\widetilde{u}(0, x_0)$ which don't converge to the true value after 10000 iterations. But the interesting thing is that deep splitting is also an effective method of solving fBSDEs and corresponding PDEs even 
without an RNN. The reason can be known from the idea and these loss functions of the deep splitting method introduced by \cite{Beck_2021}, such loss functions help us to avoid estimating
the integral of $Z$ w.r.t $B_t^H$, i.e. $\int_{}^{} Z_s \,dB_s^H$, directly.

\begin{table}
\caption{Numerical simulations for each method in the $1-d$ case of the PDE \eqref{eqn33} with $H = \frac{3}{4}(u_{true} = 6.2968)$.\label{tab2}}
\begin{tabular}{lccclc}
\hline
                            & \begin{tabular}[c]{@{}c@{}}Means of\\ $u_0$\end{tabular} & \begin{tabular}[c]{@{}c@{}}Std of\\ $u_0$\end{tabular} & \begin{tabular}[c]{@{}c@{}}Rel.$L^1$\\ error\end{tabular} & \begin{tabular}[c]{@{}l@{}}Std of\\ rel.error\end{tabular} & \begin{tabular}[c]{@{}c@{}}Avg.runtime\\ /s\end{tabular} \\ \hline
deep BSDE                 & 3.4685                                                   & 3.70e-3                                                & 0.4491                                                    & 5.88e-4                                                    & 272                                                     \\
LSTM-BSDE                 & 6.3076                                                   & 2.06e-2                                                & 3.12e-3                                                   & 1.30e-3                                                    & 430                                                      \\
mRNN-BSDE                 & 6.2970                                                   & 9.81e-3                                                & 1.11e-3                                                   & 9.38e-4                                                    & 184                                                      \\
\multicolumn{1}{c}{DS}    & 6.2687                                                   & 1.58e-2                                                & 4.57e-3                                                   & 2.51e-3                                                    & 795                                                      \\
\multicolumn{1}{c}{DBDP1} & 3.5568                                                   & 2.29e-2                                                & 0.4351                                                    & 3.64e-3                                                    & 619                                                      \\ \hline
\end{tabular}
\end{table}

\subsubsection{Results in the high-dimension case($d=50$)}
In the high-dimensional case, the fractional Black-scholes equation has the form
\begin{equation}
\label{eqn36}
\begin{cases}
\begin{aligned}
&\frac{\partial u(t,x)}{\partial t} + \sum_{i = 1}^{d}\sigma^2 H x_j^2 t^{2H-1} \frac{\partial^2 u(t,x)}{\partial x_i^2}+\sum_{i = 1}^{d}r x_i \frac{\partial u(t,x)}{\partial x_i}-r u(t,x)=0,\\
&u(T,x)=g(x),\\   
\end{aligned}    
\end{cases}
\end{equation}
where $g(X_T) = \max\left\{ \max_{1 \le i \le d}X_T^i - K , 0 \right\}$. And in this case, 
there is no known analytical solution which is different from the 1-dimensional case.

We choose $d = 50$ as the high-dimension case, $0 = t_0 < t_1 < \dots < t_N = T$, $T = 0.5$, $N = 20$, then $\Delta t = T/N = 0.025$. For the parameters, assume
$\mu = 0.06$, $\sigma = 0.2$, $x_0 = (100, 100, \dots, 100) \in \mathbb{R}^d$, $r = 0.06$, $K = 100$. Learning rate $lr = 0.008$, the valid set size $m_1 = 256$ and the mini-batch size $m_2 = 64$. To approximate $u(0,x_0)$,
there will be 5 independent runs for each of the methods.

On principle, we still try our best to keep hyperparameters same for all algorithms but we can hardly ignore the difference between different methods especially in the high-dimensional 
case. For DBDP1, we set $n = 11$ neuros on each hidden layers because if we set $n = d+10$, $\widetilde{u}(0, x_0)$ is slow to converge under $lr = 0.008$. Though we can choose to increase
learning rate for DBDP1, we have finally chosen to keep the number of neuros on each hidden layers same as the 1-dimensional case by comparing the results.

Firstly, we also make a comparison between these methods with $H = \frac{1}{2}$ in the high-dimensional case. Similar to the one-dimensional case with $H = \frac{1}{2}$,
the values of $\widetilde{u}(0, x_0)$ are all close whichever method we use as shown in Fig \ref{fig4} and Table \ref{tab3}.
\begin{figure}
\centering
\includegraphics[width=.9\textwidth]{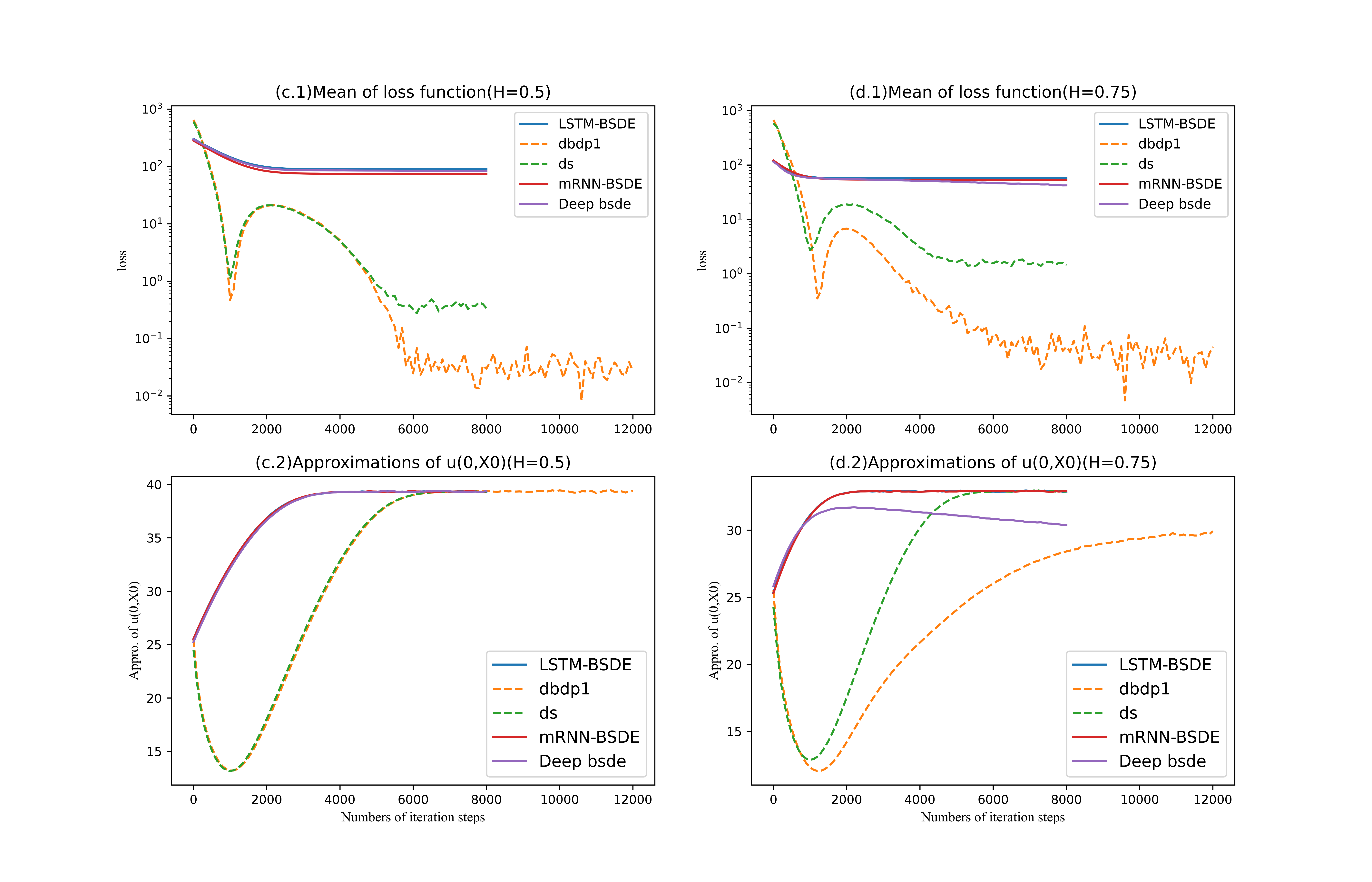}
\caption{Mean of the loss function and mean of the approximations of $u(0,x_0)$ in the $50-d$ case of the PDE \eqref{eqn36}.
(c.1) mean of the loss function when $H = 1/2$; (c.2)mean of the approximations of $u(0,x_0)$ when $H = 1/2$;
(d.1) mean of the loss function when $H = 3/4$; (d.2)mean of the approximations of $u(0,x_0)$ when $H = 3/4$.\label{fig4}}
\end{figure}

\begin{table}
\caption{Numerical simulations for each method in the $50-d$ case of the PDE \eqref{eqn36} with $H = \frac{1}{2}$.\label{tab3}}
\begin{tabular}{lcccc}
\hline
                            & \begin{tabular}[c]{@{}c@{}}Mean of\\ $u_0$\end{tabular} & \begin{tabular}[c]{@{}c@{}}Std of\\ $u_0$\end{tabular} & \multicolumn{1}{l}{\begin{tabular}[c]{@{}l@{}}Avg.time to reach\\ $\pm$0.5\% of \\ means of $u_0$/s\end{tabular}} & \begin{tabular}[c]{@{}c@{}}Avg.runtime\\ /s\end{tabular} \\ \hline
deep BSDE                 & 39.3332                                                 & 0.0254                                                 & 181                                                                                                               & 355                                                      \\
LSTM-BSDE                 & 39.3349                                                 & 0.0188                                                 & 565                                                                                                               & 1257                                                     \\
mRNN-BSDE                 & 39.3334                                                 & 0.0185                                                 & 160                                                                                                               & 316                                                      \\
\multicolumn{1}{c}{DS}    & 39.2848                                                 & 0.0160                                                 & 684                                                                                                               & 886                                                     \\
\multicolumn{1}{c}{DBDP1} & 39.3324                                                 & 0.0519                                                 & 550                                                                                                               & 1001                                                     \\ \hline
\end{tabular}
\end{table}

Set $H = \frac{3}{4}$, from Fig \ref{fig4} and Table \ref{tab4}, it can be observed that LSTM-BSDE, mRNN-BSDE and deep splitting are close to the value of $\widetilde{u}(0, x_0)$ while 
DBDP1 and deep BSDE don't offer a convergence value after at least 8000 iterations in this case. 

\begin{table}
\caption{Numerical simulations for each method in the $50-d$ case of the PDE \eqref{eqn36} with $H = \frac{3}{4}$.\label{tab4}}
\begin{tabular}{lcccc}
\hline
                            & \begin{tabular}[c]{@{}c@{}}Mean of\\ $u_0$\end{tabular} & \begin{tabular}[c]{@{}c@{}}Std of\\ $u_0$\end{tabular} & \multicolumn{1}{l}{\begin{tabular}[c]{@{}l@{}}Avg.time to reach\\ $\pm$0.5\% of \\ means of $u_0$/s\end{tabular}} & \begin{tabular}[c]{@{}c@{}}Avg.runtime\\ /s\end{tabular} \\ \hline
deep BSDE                 & NC                                                      & NC                                                     & NC                                                                                                                & NC                                                       \\
LSTM-BSDE                 & 32.9005                                                 & 2.21e-2                                                & 341                                                                                                               & 1272                                                     \\
mRNN-BSDE                 & 32.8942                                                 & 4.78e-3                                                & 100                                                                                                               & 342                                                     \\
\multicolumn{1}{c}{DS}    & 32.9015                                                 & 1.96e-2                                                & 588                                                                                                               & 881                                                    \\
\multicolumn{1}{c}{DBDP1} & NC                                                      & NC                                                     & NC                                                                                                                & NC                                                     \\ \hline
\end{tabular}
\end{table}

\subsection{Nonlinear fractional Black-Scholes equation with different interest rates for borrowing and lending($d=100$)}
Next, we gave some numerical experiments for calculating approximate solutions of some nonlinear parabolic PDEs by using mRNN-BSDE and LSTM-BSDE.
Comparing with classical linear Black-Scholes equations, nonlinear Black-Scholes equations are under more realistic assumptions and there are many types of them. Here we 
have considered a nonlinear fractional Black-Scholes equation with different interest rates for borrowing and lending, which is 
\begin{equation}
\label{eqn49} 
\begin{cases}
\begin{aligned}
& \frac{\partial u(t,x)}{\partial t} + \sum_{i = 1}^{d}\sigma^2 H x_j^2 t^{2H-1} \frac{\partial^2 u(t,x)}{\partial x_i^2}+\sum_{i = 1}^{d}\mu x_i \frac{\partial u(t,x)}{\partial x_i} \\
&+ f(t,x,u(t,x),(\nabla u\sigma)(t,x)) = 0,\\
& f(t,x,y,z) = -r^ly-\frac{\mu-r^l}{\sigma}\sum_{i = 1}^{d}z_i  + (r^b-r^l)\max\left\{0,\left[\frac{1}{\sigma}\sum_{i = 1}^{d}z_i\right]-y\right\} ,\\
\end{aligned}   
\end{cases}
\end{equation}
and 
\begin{equation}
g(x) = \max\left\{\left[\max_{1\le i \le 100}x_i\right]-120,0\right\}-2\max\left\{\left[\max_{1\le i \le 100}x_i\right]-150,0\right\}.
\end{equation}

\begin{sloppypar}
Assume $d = 100$, $T = 0.5$, $N = 20$, $\mu = 0.06$, $\sigma = 0.2$, $x_0 = (100, 100, \dots, 100) \in \mathbb{R}^d$, $r^l = 0.04$, $r^b = 0.06$, $H = 0.75$. As for the network setting,
assume learning rate $lr = 0.005$, the valid set size $m_1 = 256$ and the mini-batch size $m_2 = 64$. For both mRNN-BSDE method and LSTM-BSDE method, set one $d$-dimensional input layer,
two $d+10$-dimensional hidden layers and one $d$-dimensional output layer. To approximate $u(0,x_0)$,
there will be 5 independent runs for each of the methods. The true value of $u(0,x_0)$ of Equation \eqref{eqn49} is replaced by the reference value $u(0,x_0) = 17.0848$ which 
is calculated by deep splitting.
\end{sloppypar}

\begin{figure}
\centering
\includegraphics[width=.9\textwidth]{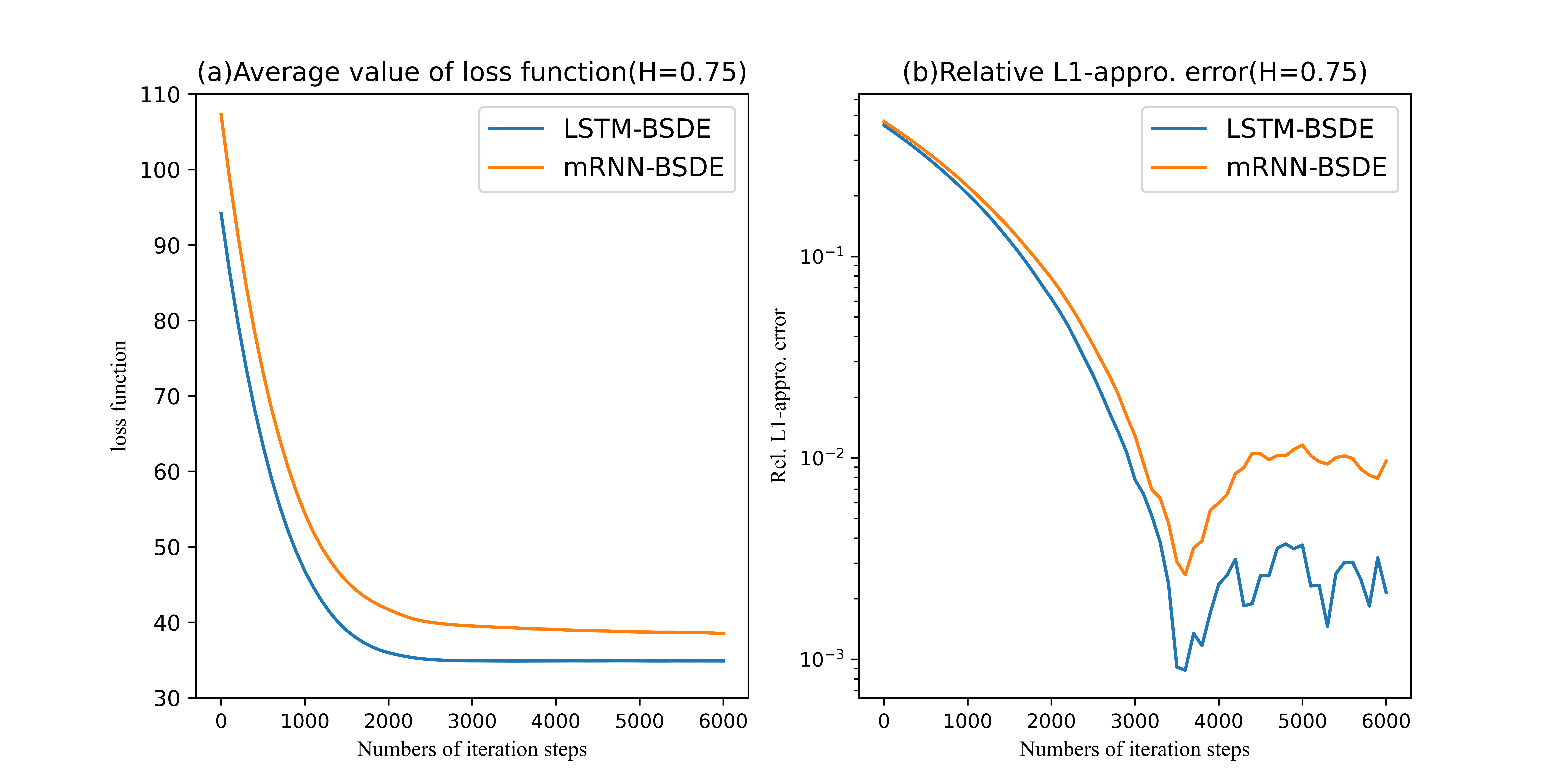}
\caption{Mean of the loss function and relative $L^1$-approximation error of $u(0,x_0)$ in the $100-d$ case of the PDE \eqref{eqn49}.
(a) mean of the loss function ; (b)relative $L^1$-approximation error of $u(0,x_0)$ when $H = 3/4$.\label{fig5}}
\end{figure}

\begin{table}
\caption{Numerical simulations for RNN-BSDE method in the $100-d$ case of the PDE \eqref{eqn49} with $H = \frac{3}{4}$.\label{tab5}}
\begin{tabular}{lccccc}
\hline
            & \begin{tabular}[c]{@{}c@{}}Mean of\\ $u_0$\end{tabular} & \begin{tabular}[c]{@{}c@{}}Std of\\ $u_0$\end{tabular} & \begin{tabular}[c]{@{}c@{}}Rel.$L^1$\\ error\end{tabular} & \begin{tabular}[c]{@{}c@{}}Std of\\  rel.error\end{tabular} & \begin{tabular}[c]{@{}c@{}}Avg.runtime\\ /s\end{tabular} \\ \hline
LSTM-BSDE & 17.0418                                                 & 1.39e-2                                                & 2.53e-3                                                   & 8.12e-4                                                     & 2732                                                     \\
mRNN-BSDE & 16.9238                                                 & 1.08e-2                                                & 9.59e-3                                                   & 6.33e-4                                                     & 479                                                     \\ \hline
\end{tabular}
\end{table}

\clearpage
\subsection{A semilinear heat equation with variable coefficients($d = 50$)}
In this subsection, we consider a type of semilinear heat equation with variable coefficients of the form
\begin{equation}
\label{eqn50}
\frac{\partial u(t,x)}{\partial t} + \sigma^2 H t^{2H-1}\Delta_x u(t,x) + \frac{1-\left\lvert u(t,x)\right\rvert^2 }{1 + \left\lvert u(t,x)\right\rvert^2} = 0,
\end{equation}
and
\begin{equation}
g(x) = \frac{5\exp\left(T^{2H}\right) }{10 + 2\left\lVert x\right\rVert^2_{\mathbb{R}^d}}.    
\end{equation} 

Assume $d = 50$, $T = 0.5$, $N = 20$, $\mu(t,x) = 0$, $\sigma(t,x) = \sigma = 1$, $x_0 = (0, 0, \dots, 0) \in \mathbb{R}^d$, $H = 2/3$ and assume $lr = 0.008$, the valid set size $m_1 = 256$ 
and the mini-batch size $m_2 = 64$. For both mRNN-BSDE method and LSTM-BSDE method, set one $d$-dimensional input layer,
two $d+10$-dimensional hidden layers and one $d$-dimensional output layer. To approximate $u(0,x_0)$,
there will be 5 independent runs for each of the methods. The true value of $u(0,x_0)$ of Equation \eqref{eqn50} is replaced by the reference value $u(0,x_0) = 0.5390$ which 
is calculated by deep splitting.

In this case, from Fig.\ref{fig6}, we can find mRNN-BSDE doesn't offer any convergence value after at least 10000 iterations 
while we get a convergence value more quickly by LSTM-BSDE.

\begin{figure}
\centering
\includegraphics[width=.9\textwidth]{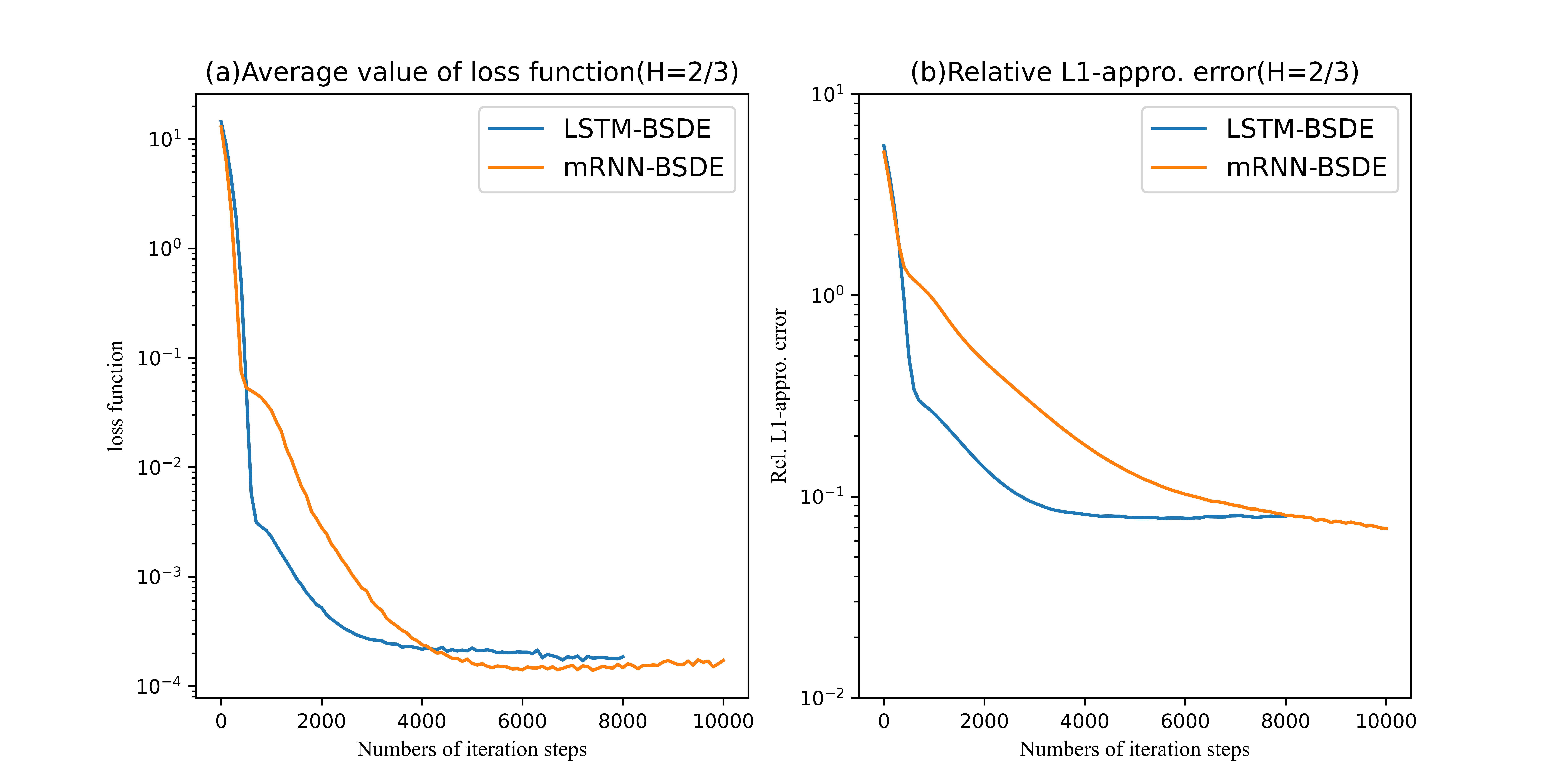}
\caption{Mean of the loss function and relative $L^1$-approximation error of $u(0,x_0)$ in the $50-d$ case of the PDE \eqref{eqn50}.
(a) mean of the loss function ; (b)relative $L^1$-approximation error of $u(0,x_0)$ when $H = 2/3$.\label{fig6}}
\end{figure}

\begin{table}
\caption{Numerical simulations for RNN-BSDE method in the $50-d$ case of the PDE \eqref{eqn50} with $H = \frac{2}{3}$.\label{tab6}}
\begin{tabular}{llcccc}
\hline
            & \multicolumn{1}{c}{\begin{tabular}[c]{@{}c@{}}Mean of\\ $u_0$\end{tabular}} & \begin{tabular}[c]{@{}c@{}}Std of\\ $u_0$\end{tabular} & \begin{tabular}[c]{@{}c@{}}Rel.$L^1$\\ error\end{tabular} & \begin{tabular}[c]{@{}c@{}}Std of\\  rel.error\end{tabular} & \begin{tabular}[c]{@{}c@{}}Avg.runtime\\ /s\end{tabular} \\ \hline
LSTM-BSDE & 0.5818                                                                     & 1.22e-3                                                & 7.93e-2                                                   & 2.27e-3                                                     & 1698                                                     \\
mRNN-BSDE & NC                                                                         & NC                                                     & NC                                                        & NC                                                          & NC                                                     \\ \hline
\end{tabular}
\end{table}
\newpage
\section{Conclusion}
Fix $H \in (\frac{1}{2},1)$, in this paper, we have discussed the relationship between the systems of PDEs and the fFBSDEs where $\int_{0}^{t} f_s \, dB_s^H$ is in the sense of a Wick integral. 
And we have given \eqref{eqn32} which is the PDE corresponding to the fFBSDEs where the solution solves the forward SDE \eqref{eqn29} is called geometric fractional Brownian motion and 
significant in finance. What's more, we have developed the RNN-BSDE method designed to solve fBSDEs. By numerical experiments, it can be observed that 
deep splitting and RNN-BSDE method are effective to solve fractional BSDEs and corresponding PDEs comparing with other methods. 
Comparing with the LSTM and stacked-RNN structure, LSTM-BSDE has a better performance. In general cases, 
the LSTM structure may be a good choice for our RNN-BSDE method regardless of time cost. 
\color{black}




\bibliographystyle{elsarticle-num} 
\bibliography{paper1-refs}






\end{document}